\documentclass[11pt, reqno, a4paper]{amsart}

\usepackage{amsmath}
\usepackage{amssymb}
\usepackage{amsthm}
\usepackage{graphicx}
\usepackage{psfrag}
\usepackage{hyperref, enumitem}

\makeatletter
\def\namedlabel#1#2{\begingroup
    #2%
    \def\@currentlabel{#2}%
    \phantomsection\label{#1}\endgroup
}
\makeatother

\allowdisplaybreaks

\newtheorem{thm}{Theorem}[section]
\newtheorem{lemma}[thm]{Lemma}
\newtheorem{prop}[thm]{Proposition}
\newtheorem{cor}[thm]{Corollary}

\newtheorem{notation}{Notation}
\newtheorem{defn}[thm]{Definition}
\newtheorem{examp}[thm]{Example}

\newtheorem{rmk}[thm]{Remark}
\newenvironment{sequation}{\begin{equation}\small}{\end{equation}}

\makeatletter
\def\blfootnote{\gdef\@thefnmark{}\@footnotetext}
\makeatother

\begin{document}

\title{Anomalous time-scaling of extreme events in infinite systems and Birkhoff sums of infinite observables}

\begin{abstract}
We establish quantitative results for the statistical be\-ha\-vi\-our of
\emph{infinite systems}. We consider two kinds of infinite system:

\begin{itemize}
\item[i)] a conservative dynamical system $(f,X,\mu)$ preserving a
  $\sigma$-finite measure $\mu$ such that $\mu(X)=\infty$;

\item[ii)] the case where $\mu$ is a probability measure but we
  consider the statistical behaviour of an observable $\phi\colon
  X\to[0,\infty)$ which is non-integrable: $\int \phi \, d\mu=\infty$.
\end{itemize}

In the first part of this work we study the behaviour of Birkhoff sums
of systems of the kind ii). For certain weakly chaotic systems,
we show that these sums can be strongly oscillating. However, if the system has
superpolynomial decay of correlations or has a Markov
structure, then we show this oscillation cannot happen. In this case we
prove a  general  relation between the  behavior of $\phi $,   the local dimension of $\mu$, and  the scaling rate of the growth of Birkhoff sums of $\phi $ as time tends to infinity. We then establish
several important consequences which apply to infinite systems
of the kind i). This includes showing anomalous scalings in extreme event limit
laws, or entrance time statistics. We apply our findings to non-uniformly
hyperbolic systems preserving an infinite measure, establishing
anomalous scalings in the case of logarithm laws of entrance times,
dynamical Borel--Cantelli lemmas, almost sure growth rates of
extremes, and dynamical run length functions.
\end{abstract}



\author{Stefano Galatolo}
\address{Stefano Galatolo, Dipartimento di Matematica, Via Buonattoti 1, 56127 Pisa, Italy}
\email{stefano.galatolo@unipi.it}

\author{Mark Holland}
\address{Mark Holland, Mathematics (CEMPS), Harrison Building (327),
  North Park Road, EXETER, EX4 4QF, United Kingdom}
\email{m.p.holland@exeter.ac.uk}

\author{Tomas Persson}
\address{Tomas Persson, Centre for Mathematical Sciences, Lund
  University, Box 118, 221 00 Lund, Sweden}
\email{tomasp@maths.lth.se}

\author{Yiwei Zhang}
\address{Yiwei Zhang, Scool of Mathematics and Statistics, Center for
  Mathematical Sciences, Hubei Key Laboratory of Engineering Modeling
  and Scientific Computing, Huazhong University of Sciences and
  Technology, Wuhan 430074, China}
\email{yiweizhang@hust.edu.cn}

\thanks{{\it Acknowledgements.} S.~Galatolo thanks GNAMPA-INdAM for partial support during this research. M.~Holland acknowledges the support by the EPSRC:
EP/P034489/1, and the support from the mathematics departments at
the University of Pisa, and the University of Houston.
T.~Persson thanks institut Mittag-Leffler and the program ``Fractal Geometry and Dynamics'', during which parts of this
paper was written. Y.~Zhang acknowledges the support by the NSFC: 11701200
and 11871262, and the support from the mathematics departments at the University of Warwick,
the University of Exeter, the University of Pisa and the Centre Physique Th\'{e}orique, Marseille.}



\maketitle

\section{Introduction} \label{sec:introduction}


\subsection{Infinite Observables}
We consider a dynamical system preserving a probability measure
$(f,X,\mu)$, together with an observable function $\phi \colon
X\rightarrow [0,\infty)$. Let us consider the case where the
observable $\phi $ is non-integrable, i.e.\ $\int \phi \,d\mu =\infty
$, and the Birkhoff sum
\begin{equation*}
S_{n}(x):=\sum_{k=0}^{n-1}\phi (f^{k}(x)).  
\end{equation*}
The pointwise ergodic theorem implies that $S_{n}(x)$ grows to
infinity faster than any linear increasing speed, for almost each
$x\in X$. For these systems, Aaronson \cite[Theorem~2.3.2]{Aaronson}
has shown that for any sequence $b(n)>0$, if $\lim_{n\rightarrow
  \infty }\frac{b(n)}{n}=\infty$ then either
\begin{equation}
  \limsup_{n\rightarrow \infty }\frac{S_{n}(x)}{b(n)}=\infty \quad
  \text{a.e.}  \qquad \text{or}\qquad \liminf_{n\rightarrow \infty
  }\frac{S_{n}(x)}{b(n)} =0\quad \text{a.e.}  \label{aar}
\end{equation}
Thus, for these kind of systems a kind of pointwise ergodic theorem
cannot hold for the asymptotic behaviour of the ratio
$\frac{S_{n}(x)}{b(n)}$ for every possible rescaling sequence $b(n)$.\footnote{A similar phenomenon occurs for $L^1$ observables in systems preserving an infinite measure, see \cite{Aaronson}, or \cite{LeMu} for a discussion and recent developments.} It is then natural to investigate the speed
of growth of such Birkhoff sums quantitatively from a coarser point of view. We approach this
problem in the first part of the paper finding  general estimates on the scaling behaviour of  such Birkhoff sums  growth. In the second part of the
paper we consider applications of these studies to understand several
quantitative ergodic features of systems preserving an infinite
measure. We set up a general framework and give examples of
application to a family of intermittent, non uniformly hyperbolic maps,
finding a kind of anomalous time-scaling for several quantitative
statistical properties of the dynamics related to extreme events and
hitting times.  The understanding of the asymptotic behaviour of
Birkhoff sums of infinite observables also has other important
applications. We mention as an example the works \cite{Cori, Cori2} where this is used to estimate the speed of mixing on
area preserving flows on surfaces.

To obtain estimates from above on the behaviour of $S_{n}$, the following
general result is useful.
\begin{prop}[{Aaronson {\protect\cite[Proposition~2.3.1]{Aar2}}}] \label{aa}
  If $a(x)$ is increasing, $\lim_{x\rightarrow \infty} \frac{a(x)}{x}
  = 0$ and
  \begin{equation*}
    \int a(\phi (x)) \, d\mu (x) <\infty,
  \end{equation*}
  then for $\mu$-a.e.\ $x\in X$
  \begin{equation*}
    \lim_{n\rightarrow \infty }\frac{a(S_{n})}{n}=0.
  \end{equation*}
\end{prop}

\begin{rmk} \label{aarmk}
  Consider the case where $\phi(x)=d(x,\tilde{x})^{-k}$ for some
  $\tilde{x} \in X$ and $k\geq 0$, and denote by $d_{\mu }(\tilde{x})$
  the local dimension of $\mu $ at $\tilde{x}$. From
  Proposition~\ref{aa}, if we let
  $a(x)=x^{\frac{d_{\mu}(\tilde{x})}{k}-\varepsilon_1 }$ for some
  $\varepsilon_1 >0$ we get $\int a(\phi(x)) \,d\mu(x) <\infty $. This
  implies that for each $\varepsilon >0$ and $\mu$-almost all $x$, we
  have eventually (as $n\to\infty$)
  \begin{equation}
    S_{n}(x)\leq n^{\frac{k}{d_{\mu }(\tilde{x})}+\varepsilon }.  \label{eq2}
  \end{equation}
\end{rmk}

As it will be shown in Proposition~\ref{the:thereareoscillations} and
in Section~\ref{oscill}, there are systems for which the asymptotic
behaviour of $S_{n}$ is strongly oscillating, or far from the estimate
given in \eqref{eq2}.

Thus, establishing convergence (or finding the typical growth rate) of $S_n$ is
in general non-trivial, and suitable assumptions are needed on the
system to get a definite asymptotic behaviour for $S_{n}$. Lower bound
estimates on the growth rates of $S_n$ have been given in
\cite{CN17,G2} under assumptions related to hitting time statistics
and recurrence. These assumptions include having a logarithm law for
the hitting time or a dynamical Borel--Cantelli property for certain
shrinking target sets. We now review these connections in greater
detail.

\subsubsection{Known relations between Birkhoff sums of infinite observables, hitting times and Borel--Cantelli properties} 
Our first main result, Theorem~\ref{thm1} establishes almost sure
bounds on the growth rate of $S_n$ under mild assumptions on the
system $(f,X,\mu)$ and the non integrable observable function
$\phi$. We include the case where the system $(f,X,\mu)$ has
super-polynomial decay of correlations and allow the obserbable
$\phi(x)$ to be quite general. We only impose a regularity assumption
on the level sets $\{\phi(x)=u\}$. In particular our results allow for
the fact these sets might not be homeomorphic to balls (in a given
Riemannian metric), e.g.\ $\{\phi(x)\geq u\}$ might be a \emph{tube}
or another regular set.\footnote{Growth of sums, and related hitting
  statistics for these types of observables has been the focus of
  recent interest, see \cite{FFT2, G, HVRSB, V.et.al}.}

Let us now briefly discuss the hitting time scaling behaviour
indicators considered in \cite{G2} and their relation with $S_n$.  Let
$B(\tilde{x},r)$ be a ball with centre $\tilde{x}$ and radius $r$.  We
define the first hitting (or entrance) time of the orbit of $x$ to
$B(\tilde{x},r)$ by
\begin{equation*}
  \tau _{r}(x,\tilde{x}) := \min \{\,n\in
  \mathbb{N}:n>0,\ f^{n}(x)\in B(\tilde{x},r)\,\}\,.
\end{equation*}
Then define the hitting time indicators as
\begin{equation*}
  \overline{H}(x,\tilde{x}) := \limsup_{r\rightarrow 0}\frac{\log
    (\tau _{r}(x,\tilde{x}))}{-\log (r)}, \qquad
  \underline{H}(x,\tilde{x}) := \liminf_{r\rightarrow 0}\frac{\log
    (\tau _{r}(x,\tilde{x}))}{-\log (r)}.
\end{equation*}

To help in understanding the sense of these definitions, we remark that according to the definitions, $ \tau _{r}(x,\tilde{x}) $ scales like $r^{{H}(x,\tilde{x})}$.
If observables of the form $\phi (x)=d(x,\tilde{x})^{-k}$ are
considered then relations between $\overline{H},$ $\underline{H}$ and
the behaviour of Birkhoff sums of infinite observables are proved in
\cite{G2}. Among these, it is shown that for each $\varepsilon >0$,
eventually ($n\to\infty$),
\begin{equation*}
  S_{n}(x)\geq n^{\frac{k}{\overline{H}(x,\tilde{x})}-\varepsilon }  
\end{equation*}
holds $\mu$-a.e. We recall that $\overline{H}$ and $\underline{H}$
have been estimated in many systems (see e.g.\ \cite{G3, GK, GP, GP2,
  G4, GSR, kimseo} and references therein) and are related to the
local dimension of the invariant measure in strongly chaotic systems,
while in weakly chaotic or non chaotic ones they also have relations
with the arithmetical properties of the system. In particular it is
proved that in fastly mixing systems
\begin{equation*}
  \overline{H}(x,\tilde{x})=d_{\mu }(\tilde{x})
\end{equation*}
holds for $\mu$-a.e.\ $x$ (see Proposition~\ref{prop.G} for a precise
statement) hence implying for $\mu$-almost every $x$, the lower
estimate
\begin{equation*}
  S_{n}(x)\geq n^{\frac{k}{d_{\mu }(\tilde{x})}-\varepsilon }
\end{equation*}
holds for the observable $\phi (x)=d(x,\tilde{x})^{-k}$, and for large $n$
(compare with $($\ref{eq2}$)$).

In the recent paper \cite{CN17}, it is supposed that the system has
absolutely continuous invariant measure, on a space of dimension $D\in
\mathbb{N}$ and to satisfy a strong Borel--Cantelli
assumption.\footnote{The reader can find in Section~\ref{BC} precise
  definitions about the strong or weak Borel--Cantelli assumption.  We
  remark that these assumptions are strictly related to the hitting
  time behaviour, as it is shown in \cite{GK}.} Under this assumption,
it is shown that for each $\varepsilon >0$ and $\mu$-almost every $x$,
for the observables of the kind $\phi (x)=d(x,\tilde{x})^{-k}$ with
$k\geq 0$, we have eventually ($k\to\infty$)
\begin{equation*}
  n^{\frac{k}{D}-\varepsilon }\leq S_{n}(x)\leq n^{\frac{k}{D}}(\log
  n)^{\frac{ k}{D}+\varepsilon}.
\end{equation*}
Other similar results are given in the case the system is exponentially
mixing and the invariant measure has density in $L^{p}$, or in particular
cases of intermittent maps.

\subsubsection{Growth of Birkhoff sums and extremes.}\label{sec.sum-max}
Given a measure preserving system $(f,X,\mu)$ consider the
maximum process
\begin{equation}\label{eq.max-process}
  M_n(x):=\max_{0\leq k \leq n-1}\phi(f^{k}(x)),
\end{equation}
where $\phi \colon X\to\mathbb{R}$ is an observable function.  In the
case where $\phi\geq 0$ on all of $X$, it is clear that $S_n(x) \geq
M_n(x)$.  Hence $M_n(x)$ can provide a lower bound for $S_n(x)$.  In
\cite{Embrechts} it is proved that if a process $(X_{n})$ is generated
by i.i.d.\ random variables, and $\|X_1\|_1<\infty$ then
$M_{n}/S_{n}\to 0$ ($\mu$-almost surely).  Conversely, in the case of
infinite observables the behaviour of $M_{n}$ gives good lower bounds
in many interesting systems, approaching the general upper bound given
in Proposition~\ref{aa}. This is indeed the strategy used to get lower
bounds to $S_{n}(x)$ in \cite{CN17, G2} and in the present paper to
get Theorem~\ref{thm1}.

In the classical probabilistic literature the statistical properties
of such $M_n$ are of interest to those working in \emph{extreme value
  theory}, \cite{Embrechts, Galambos}.  For dynamical systems
preserving a probability measure, the distributional properties of
$M_n$ are known in some cases (e.g.\ \cite{V.et.al}, \cite{FFT1}). For
certain dynamical systems, almost sure growth rates of $M_n$ have also
been investigated \cite{G2,GNO,HNT}.  In this article, we give precise
quantification on the almost sure behaviour of $M_n$ for a general
class of infinite observables. The process $M_n$ is indeed strongly
related to the hitting time $\tau_r(x,\tilde{x})$. In the case
$\phi(x)=\psi(d(x,\tilde{x}))$, for some monotone decreasing function
$\psi \colon [0,\infty)\to\mathbb{R}$, then the event
\begin{equation} \label{ev}
  \{M_n(x)\leq u\}
\end{equation}
corresponds to the event $\{\tau_{r(u)}(x,\tilde{x})\geq n\}$ with
$r(u)=\psi^{-1}(u)$. Hence all three processes $S_n(x), M_n(x),
\tau_{r}(x,\tilde{x})$ are interlinked. This allows us to transfer
(almost sure) statistical information from any one of these processes,
to the other two.  The relation between $M_n(x)$ and
$\tau_r(x,\tilde{x})$ is explained in a very general setting and in
more detail in Section~\ref{sec.max-hit}. This construction allows us
to establish new results (Theorem~\ref{thm1} and Proposition~\ref{prop:SnR}) on the almost sure growth of $M_n$ for a more
general class observables (not related to the distance from a given point) with respect to those considered in
e.g.\ \cite{CN17, G2, HNT}  and to systems having an invariant measure which is not absolutely continuous (including the
case of measures having a non integer local dimension). In particular these results have relevance to the
case where $\phi$ is a \emph{physical observable}, see \cite{HVRSB, V.et.al}.



\subsection{Systems preserving an infinite measure}

Based on the findings on the behaviour of Birkhoff sums and maxima of
an infinite observable, we are able to address a number of relevant
topics relating to systems preserving an infinite measure. We
formulate a general framework, and show application to the the
celebrated family of ``intermittent'' maps studied by Manneville--Pomeau in
\cite{MP}, and by Liverani--Saussol--Vaienti in \cite{LSV}.
We focus on the case where $(f,X,\mu ) $ is conservative, ergodic, and
$\mu $ is $\sigma $-finite, with $\mu (X)=\infty $.  The main idea
here is to analyse a map induced over a finite part of the infinite
system.  The dynamical behaviour of the finite induced system is then
easier to study and the findings can be applied to the original
system, which can be seen as a suspension of the induced one (the
construction is outlined at the beginning of Section~\ref{extremes}).
The suspension in our case will have an associated infinite observable
which plays the role as the ``return time function.'' The results
motivated in the previous sections give important information in this
construction, such as understanding the Birkhoff sums of this
observable.\footnote{We remark that the return time function can be an
  observable with quite a complicated structure, not necessarily
  related to the distance from a point. Thus to get information on its
  Birkhoff sum, the conclusion of Theorem~\ref{thm1} is important as
  it extends to general observables.}  We find that the behaviour of
the infinite observable gives a kind of ``time rescaling'' factor
which is important in the behaviour of several quantitative ergodic
aspects of the dynamics of such infinite systems. In particular we
find this ``anomalous scaling'' in the following aspects:

\subsubsection*{The behaviour of the hitting time to small targets, and logarithm
  laws.}

Here we are interested in the time needed for a typical trajectory of
the system to hit a small target which could be seen as an extreme
event. Let $A_{n}$ be a sequence of targets of measure going to zero
and consider the hitting time to the $n$-th target
\begin{equation*}
  \tau(f,x,A_n):=\inf \{\, n\geq 0: f^{n}(x)\in A_{n} \,\}.
\end{equation*}
It is proved (see \cite{G2,G3,G,GN,GP,GP2} and references therein)
that in a wide variety of systems preserving a probability measure,
a logarithm law holds
\begin{equation}
  \lim_{n\rightarrow \infty }\frac{\log (\tau(f,x,A_n))}{-\log (\mu
    (A_{n}))}=1,
  \label{ll}
\end{equation}
provided the target sets $A_n$ are regular enough, and the system is sufficiently chaotic
(a precise statement of this kind is shown in
Proposition~\ref{pro1}). For infinite systems having a fast mixing
first return map on a finite subspace, we show that the ratio in
\eqref{ll} converges to a number $\alpha $ that depends on the return
time associated to the return map. Hence we obtain an anomalous
behaviour in a wide class of infinite systems (see
Proposition~\ref{start} for a precise statement).
We remark that a similar anomalous scaling was already found in quantitative recurrence indicators, in \cite{GKP}.

\subsubsection*{Almost sure scaling laws for the statistics of extremes in systems
  preserving an infinite measure.}

For infinite systems $(f,X,\mu)$ which are conservative, ergodic and $\mu$ is
$\sigma$-finite, we consider
the behaviour of maxima $M_{n}$ of a given observable $\phi$, see
\eqref{eq.max-process}. As discussed in the Section~\ref{sec.sum-max}
(see \eqref{ev}), this is naturally related to the hitting times.  In
Proposition~\ref{poi} we show a precise, quantitative link between
maxima and hitting time behaviour. As a consequence we obtain an
estimate for the behaviour of $M_{n}$ in infinite systems (see
Corollary~\ref{t2}).  As obtained for the hitting time problems and
logarithm laws, we show that the scaling behaviour of $M_{n}$ depends
on the return time statistics associated to the infinite measure
system (in a way we make precise in Section~\ref{sec:mainresults}), as
well as on the local regularity of the observable function $\phi
$. This is unlike the behaviour of $M_{n}$ in the probability measure
preserving case. This will be done in Section~\ref{sec.bcmaxhit}. We
apply our theory to a family of intermittent maps in
Section~\ref{sec.intermittent}.

\subsubsection*{Dynamical Borel--Cantelli laws for infinite measure
  preserving systems.}

Consider a measure preserving dynamical system $(f,X,\mu )$, and let
$(\phi_{n})$ be a sequence of observables. Furthermore let
$U_{n}(x):=\sum_{k=0}^{n-1}\phi _{k}(f^{k}(x))$, with $\mu (\phi
_{k}):=\int \phi _{k}\,d\mu $. Now suppose that $\mu (\phi
_{k})\rightarrow 0$, but $\sum_{k}\mu (\phi _{k})=\infty $.  A
\emph{dynamical Borel--Cantelli problem} is the problem to show
existence (or otherwise) of a sequence $a_{n}\rightarrow \infty$ with
$U_{n}(x)/a_{n}\rightarrow 1$, $\mu$-almost surely. In the case $\int
\phi _{k}\,d\mu =c$, we are just in a strong law of large numbers type
of situation. Hence, we aim to generalise this concept in a
non-stationary setting, i.e.\ where the observable $\phi$ changes with
time.  We address this problem for the system $(f,X,\mu)$, where $\mu$
is a $\sigma$-finite (infinite) invariant measure. In the probability
preserving case, this problem has been widely studied, and forms the
basis of \emph{dynamical Borel--Cantelli Lemma} results, see
\cite{Gouezel, GNO, HNPV, Kim}. For such systems it is shown that
$a_{n}=\sum_{k=0}^{n-1}\mu (\phi _{k})$ is the typical scaling law,
and this is consistent with the corresponding theory for
i.i.d.\ random variables, see \cite{Embrechts,Galambos}. For infinite
systems, we show that this scaling sequence is not the appropriate one
to use, and we derive the corresponding scaling law. Such a result is
new, and we apply our methods to obtain shrinking target
(Borel--Cantelli) results for the intermittent map family described in
\cite{LSV} for the $\sigma $-finite (infinite) invariant measure case.
As a further application we consider infinite systems modelled by
Young towers, see Section~\ref{sec:youngtowers}.

\subsubsection*{Dynamical run-length problems for infinite measure
  preserving systems.}

Suppose further that the measure preserving dynamical system
$(f,X,\mu)$ admits a countable or finite partition
$\{I_{j}\}_{j\in\mathcal{I}}$ on $X$ (with $\mathcal{I}\subset\mathbb{N}$ an index set),
and each $x\in X$ is coded with
the sequence $(\varepsilon_{k}(x))_{k=1}^{\infty}$, by
$\varepsilon_{k}(x)=j$ if and only if $f^{k-1}(x)\in I_{j}$. The
\emph{dynamical run length function of digit $j$} is defined by
\begin{sequation}\label{def_runlengthfunction}
  \xi^{(j)}_{n}(x) := \max \{0 \leq k \leq n : \exists 0\leq i\leq
  n-k,\varepsilon_{i+1}(x) = \cdots = \varepsilon_{i+k}(x)=j \}.
\end{sequation}
In the setting of successive experiments of coin tossing, $\xi_{n}$
corresponds to the longest length of consecutive terms of
``heads/tails'' up to $n$-times experiments \cite{Eroren70,deMo1738}.
Thus, the studies of dynamical run length functions is concerned with
quantifying the asymptotic growth behaviour of $\xi_{n}(x)$ for
$\mu$-typical $x$. Such studies admit various applications in DNA
sequencing \cite{AGW90}, finance and non-parametric statistics
\cite{Bal01,Bar92,Bat48,Mus00}, reliability theory \cite{Mus00},
Diophantine approximation theory to $\beta$-expansions of real numbers
\cite{BL14,FW12,Tonyuzhao}, and Erd\H{o}s--R\'{e}nyi strong law of large numbers
\cite{denzab07,dennic13,Eroren70,Gri93}.


We analyse the dynamical run length function in the case where
$(f,X,\mu)$ is conservative, ergodic, and $\mu$ is $\sigma $-finite.
In particular, we explicitly estimate the growth rate for $\xi_{n}(x)$
for a family of intermittent maps in the $\sigma$-finite measure case
in Section~\ref{sec:runlength}.  In contrast to probability measure
preserving systems (e.g.\ uniformly hyperbolic Gibbs--Markov systems;
logistic-like maps satisfying the Collet--Eckmann condition; and
families of intermittent maps preserving
a.c.i.p.\ \cite{CFZ18,denzab07,dennic13,Gri93,Tonyuzhao}), we show
that apart from the local dimension, there is an additional scaling
contribution, arising from the asymptotics of the return time function
associated to the induced transformation, which needs to be taken into
account in the growth rate for $\xi_{n}$ of infinite systems.  As the
reader will realize, our proof is based on a natural link between the
dynamical run length function, hitting time, and growth of maximum for
the return time functions.  We are not aware of any such links,
previously established in the literature of this subject.

\subsection{Outline of the paper.}
We structure the paper as follows. In Section~\ref{sec:mainresults} we
state the main theoretical results. This includes results on the
growth of Birkhoff sums for rapidly mixing systems, on the link
between hitting time laws and growth of extremes, and dynamical
Borel--Cantelli Lemma results for systems preserving a $\sigma$-finite
infinite measure. In Sections~\ref{sec:proofsuperpolynomial} and
\ref{sec:proofinfiniteSBC} we prove these results, and then consider
several independent topics which relate to our theory. This includes a
result on the almost sure growth rates of extremes and hitting times
for infinite systems, see Section~\ref{sec.bcmaxhit}. We then apply
our theory to an intermittent map case study in
Section~\ref{sec.intermittent}, which includes a study of dynamical
run-length problems in Section~\ref{sec.runs}. We then describe
situations in which the Birkhoff sums can wildly oscillate in
Section~\ref{oscill}.  Finally we consider Borel--Cantelli results for
general Markov extensions, such as Young towers
(Section~\ref{sec:youngtowers}).

\section{Statement of main results}

\label{sec:mainresults}

\subsection{Birkhoff sums, maxima, and hitting time statistics}

\label{sec:hittingtime} Consider a dynamical system preserving a
probability measure $(f,X,\mu )$, together with an observable function
$\phi \colon X\rightarrow [0,\infty)$ with $\int \phi \, d\mu =\infty
  $. Let us recall the notation used for Birkhoff sums and
  respectively maxima of an observable $\phi$.
\begin{equation*}
  S_{n}(x):=\sum_{k=0}^{n-1}\phi (f^{k}(x)),\qquad M_{n}(x):=\max_{0
    \leq k\leq n-1}\phi \bigl( f^{k}(x) \bigr).
\end{equation*}
In specific contexts, we sometimes emphasize the dependence on $\phi$,
and write $S_{n}^{\phi}(x)$ for $S_n(x)$ (and similarly for maxima).

As noted before (see $(\ref{aar})$) it is impossible to get precise
estimates for the asymptotic behaviour of $S_{n}$ as $n$ increases.
However, under suitable assumptions on ergodicity and on the chaotic
properties of the system, coarser estimates on asymptotic growth rates are
possible.

We show that we can achieve estimates for the scaling behaviour of
both $S_{n}$ and $M_{n}$ for systems which are super-polynomially
mixing, and for quite a large class of observables having some
regularity.  The regularity we need is a kind ``Lipschitz'' regularity
of the suplevels of the observable $\phi $. This is explained in the
next definition.  Essentially we ask that the suplevels of $\phi$ are
regular enough that they could be sublevels of a Lipschitz function.

\begin{defn}
\label{reg} Let $\phi \colon X\rightarrow [0,\infty)$ be an unbounded
function. Consider the suplevels of $\phi$, defined by
\begin{equation*}
A_{n}=\{\,x\in X:\phi (x)\geq n\,\}.
\end{equation*}
We say that $\phi$ has regular suplevels if the following holds:

\begin{enumerate}
\item[(i)] We have $\lim_{n\rightarrow \infty }\mu (A_{n})=0$, there is a
constant $c>0$ satisfying $\mu (A_{n+1})>c\mu (A_{n})$
eventually as $n$ increases, and there is $\alpha _{\phi }\in [0,\infty)$, such that
\begin{equation*}
\alpha _{\phi }=\lim_{n\rightarrow \infty }\frac{\log \mu (A_{n})}{-\log (n)}.
\end{equation*}

\item[(ii)] There is $\beta \geq 0$ and a Lipschitz function $\tilde{\phi}%
\colon X\rightarrow \mathbb{R}^{+}$ such that
\begin{equation*}
A_{n}=\{\,x\in X:\tilde{\phi}(x)\leq (\mu (A_{n}))^{\beta }\,\}.
\end{equation*}
\end{enumerate}
\end{defn}

This assumption is verified by a large class of observables, including observables related to the distance from a point.

\begin{examp}
  Suppose $X$ be a Riemannian manifold with boundary, and $d(.,.)$ the Riemannian distance.
For $\tilde{x}\in X$ consider an observable of
  the form $\phi (x)=d(x,\tilde{x})^{-k}$ with $k\geq 0$ in a
  neighbourhood of $\tilde{x}$, and suppose $d_{\mu }(\tilde{x})$
  exists and $d_{\mu }(\tilde{x})>0$. Such conditions are verified for almost each
  $\tilde{x}$ in a wide class of uniformly and non-uniformly
  hyperbolic systems, see \cite{pes}. Then we have that $A_{n}$ is a
  ball of radius $r_{n}=\frac{1}{\sqrt[k]{n}}$, and hence a regular set. In this case
  \begin{equation*}
    \alpha_\phi :=\lim_{n\rightarrow \infty }\frac{\log \mu (\{x\in
      X:\phi (x)\geq n\})}{-\log (n)}=\frac{d_{\mu }(\tilde{x})}{-k},
  \end{equation*}
  and for each $\epsilon >0$, eventually $n^{\frac{d_{\mu
      }(\tilde{x})+\epsilon }{ -k}}\leq \mu (A_{n})\leq
  n^{\frac{d_{\mu }(\tilde{x})-\epsilon }{-k}}.$ Consider $\beta $
  such that $\beta d_{\mu }(\tilde{x})>1$,
  and $\tilde{\phi}(x)$ defined as
  \begin{equation*}
    \tilde{\phi}(x)=\mu (A_{i})^{\beta }+\frac{\mu (A_{i-1})^{\beta
      }-\mu (A_{i})^{\beta }}{r_{i-1}-r_{i}}(d(\tilde{x},x)-r_{i}),
  \end{equation*}
  when $d(\tilde{x},x)\in \lbrack r_{i},r_{i-1})$. Since $\frac{\mu
    (A_{i})^{\beta }}{r_{i}}\leq i^{\frac{\beta (d_{\mu
      }(\tilde{x})-\epsilon )-1}{-k}}$ is bounded for our choice of
  $\beta $, the function $\tilde{\phi}$ is Lipschitz.
\end{examp}

Other examples include cases where the suplevels correspond to tubes
or other sets, see \cite{G3, GN} for results about hitting times on
targets which are suplevels of a Lipschitz function, applied to the
geodesic flow, in which the targets relate to ``cylinders'' in the
tangent bundle instead of balls.  Thus the regular sublevels
assumption of Definition~\ref{reg} holds for a wide class of dynamical
systems and observable geometries.  We now consider the notion of
decay of correlations.

\begin{defn} \label{decay-cor}
  Let $\mathcal{B}$ be a Banach space of functions from $X$ to
  $\mathbb{R}$. A measure preserving system $(f,X,\mu )$ is said to
  have decay of correlations in $\mathcal{B}$ with rate function
  $\Theta(n)$, if for $\varphi,\psi\in\mathcal{B}$, we have
  \begin{equation*}
    \biggl |\int \varphi \circ f^{n}\psi \,d\mu -\int \varphi \,d\mu
    \int \psi \,d\mu \biggr|\leq \lVert \varphi \rVert \lVert \psi
    \rVert \Theta (n).
  \end{equation*}
  Here $\lVert \cdot \rVert $ stands for the norm on $\mathcal{B}$.
\end{defn}

Usually decay of correlations is proved for a particular space $\mathcal{B}$,
and a specified rate $\Phi(n)$.

\begin{defn}[Condition (SPDC)] \label{sup}
  We say that $(f,X,\mu)$ satisfies condition (SPDC) (Super-Polynomial
  Decay of Correlations) if for all $\alpha >0$, we have
  $\lim_{n\to\infty} n^{\alpha}\Theta (n)=0,$ and $\mathcal{B}$ is the
  space of Lipschitz continuous functions.
\end{defn}

This condition is quite general, and many systems having some form of
piecewise hyperbolic behaviour satisfy it. See \cite{baladi} for a
survey containing a list of classes of examples having exponential or
stretched exponential decay of correlations. When this kind of decay
holds for Lipschitz observables these examples satisfy (SPDC). We
remark that if a system has a certain decay of correlations with
respect to H{\"o}lder observables, then it will have the same or faster
speed when smoother observables (such as Lipschitz ones) are considered.

Now, suppose the non-integrable observable $\phi $, has regular suplevels as
in Definition~\ref{reg}. The following theorem concerns the growth of maxima
and Birkhoff sums of $\phi $.

\begin{thm}
\label{thm1} Let $(X,f,\mu )$ be a probability measure preserving system on
a metric space $X$, satisfying condition (SPDC). Let $\phi $, $A_{n}$
and $ \alpha =\alpha _{\phi }$ be as in Definition~\ref{reg}, with
$\Vert \phi \Vert _{1}=\infty $. If $\alpha >0$, then for each
$0<\varepsilon <\alpha $ and $\mu $-a.e.\ $x\in X$, there exists $N\in
\mathbb{N}$ such that for all $n\geq N$,
\begin{equation*}
n^{\frac{1}{\alpha +\varepsilon }}\leq M_{n}(x)<S_{n}(x)\leq n^{\frac{1}{%
\alpha -\varepsilon }}.
\end{equation*}%
If $\alpha =0$, then for $\mu$-a.e.\ $x\in X$
\begin{equation*}
\lim_{n\rightarrow \infty }\frac{\log (S_{n}(x))}{\log n}=\lim_{n\rightarrow
\infty }\frac{\log (M_{n}(x))}{\log n}=\infty.
\end{equation*}
\end{thm}
This theorem therefore applies to a wide class of observable geometries.
In the case where $\phi $ is related
to the distance from a point, or for particular dynamical systems, see \cite{G2, G4} or
\cite[Theorem~2.5]{HNT}, or \cite[Section~2.3]{CN17}.
Paarticular systems that are captured by the theory include H\'enon maps
\cite{BY}, and certain Poincar\'e return maps for Lorenz attractors \cite{GP2}, to name
a few. The proof of Theorem~\ref{thm1} can be found
in Section~\ref{sec:proofsuperpolynomial}.

\begin{notation}
A statement of the form $v (n) \sim u(n)$ as $n \to \infty$ means that there
is a constant $c > 0$ such that
\begin{equation*}
c^{-1} \leq \frac{v(n)}{u(n)} \leq c
\end{equation*}
holds for all large enough $n$.
\end{notation}

\begin{rmk}
  In the above setting, if we make the stronger assumption $\log \mu
  (A_{n})\sim n^{-\alpha }$ we can get the more precise upper estimate
  that $S_{n}(x)\leq n^{\frac{1}{\alpha }}(\log n)^{\frac{1}{\alpha
    }+\varepsilon }$ holds eventually for every $\varepsilon >0$ and
  $\mu$-a.e.\ $x\in X$.
\end{rmk}

\subsubsection{Gibbs--Markov systems.}


Theorem~\ref{thm1} shows how, with some strong assumptions on the system, we
can get information on the scaling behaviour of $S_n$. We will see
(Proposition~\ref{prop:SnR}) that if we assume some even stronger
assumptions on the system, as the presence of a Gibbs--Markov structure, we
can get even more precise estimates.

Consider again a transformation $(f,X,\mu )$ and an observable $\phi \colon
X\rightarrow \lbrack 0,\infty )$ with $\int \phi \,d\mu =\infty $. We say
that $(f,X,\mu )$ is a Gibbs--Markov system \cite{Aaronson} if we have the
following set up.

\begin{enumerate}
\item[\namedlabel{G-M1}{(A1)}] $X$ is an interval and there is a
  countable Markov partition $\mathcal{P}=\{\, X_i : i\in\mathbb{N}
  \,\}$ such that $f (X_i)$ contains a union of elements of $\mathcal{P}$,
and there exists $c_0>0$ such that $|f(X_i)|>c_0$.
Let $\mathcal{P}_n = \bigvee_{i=0}^{n}
  f^{-i}\mathcal{P}$.
\item[\namedlabel{G-M2}{(A2)}] There exists $\lambda>1$, $C>0$ such that for
  all $x\in X$ we have $|(f^{n})'(x)|\geq C\lambda^n$ for $n \geq 1$.
\item[\namedlabel{G-M3}{(A3)}] Uniform bounded distortion estimates hold on
  $\omega\in\mathcal{P}_n$.  That is, there exist $0 < \tau < 1$, $C>0$,
  such that for all $x,y\in\omega$, and
  $\forall\omega\in\mathcal{P}_n$,
  \[
  \left|\log\left(\frac{|f'(x)|}{|f'(y)|}\right)\right|
  \leq C \tau^{n}.
  \]
\item[\namedlabel{G-M4}{(A4)}] The measure $\mu$ is the unique
  invariant probability measure which is absolutely continuous with
  respect to Lebesgue measure.
\end{enumerate}

For a Gibbs--Markov system $(f,X,\mu)$ satisfying
\ref{G-M1}--\ref{G-M4}, we will use the following assumptions on the
observable $\phi$.
\begin{enumerate}
\item[\namedlabel{G-M5}{(A5)}] For any $X_i \in \mathcal{P}$ the
  restriction $\phi |_{X_i}$ of $\phi$ to $X_i$ is constant.
\item[\namedlabel{G-M6}{(A6)}] The observable $\phi$ satisfies the
  following asymptotics: there exists $\beta \in (0,1)$ such that
  \[
  \mu\{\, x\in X: \phi (x) = n \,\}\sim n^{-\beta-1}.
  \]
\end{enumerate}
\begin{rmk}
Note that assumption~\ref{G-M6} implies that the observable $\phi$ is
non-integrable, $\int \phi(x)\,d\mu = \infty$.
\end{rmk}

If $(f,X,\mu )$ is a Gibbs--Markov maps satisfying
assumptions~\ref{G-M1}--\ref{G-M6}, we are able to obtain a result on
the asymptotic speed of the typical growth of Birkhoff sums of $\phi$,
which we will now state. Our result below is similar to a result by
Carney and Nicol \cite[Theorem~4.1] {CN17}, and the proofs are also
similar. Carney and Nicol assumed that the system satisfies a strong
Borel--Cantelli lemma, but we do not assume this explicitly.

\begin{prop} \label{prop:SnR}
  Let $(f,X,\mu )$ satisfy assumptions~\ref{G-M1}--\ref{G-M6} , with
  $\beta $ as in assumption~\ref{G-M6}. Then each $\varepsilon >0$,
  and $\mu$-almost all $x\in X$ there is an $n_{0}$ such that for all
  $n>n_{0}$,
  \begin{equation*}
    n^{\frac{1}{\beta }}(\log n)^{-\frac{1}{\beta }-\varepsilon }\leq
    M_{n}(x)<S_{n}(x)\leq n^{\frac{1}{\beta }}(\log n)^{\frac{1}{\beta
      } +\varepsilon }.
  \end{equation*}
\end{prop}

\subsubsection{Oscillating Birkhoff sums}\label{oscillsec}

Proposition~\ref{aa} gives us a general upper bound on the increase of
$S_{n}$.  It does not depend on quantitative properties of the
dynamical system but appears to be near to an optimal estimate in many
strongly chaotic systems, see \cite{CN17} for a discussion. However,
there are chaotic systems for which the bound we obtain from
\eqref{eq2} is far from the actual behaviour of $S_{n}$, and there are
examples in which their Birkhoff sum is \emph{strongly
  oscillating}. A family of such  examples take the form of a skew product map
$f\colon[0,1]\times S^{1}\rightarrow \lbrack 0,1]\times S^{1}$ defined by
\begin{equation}
  f(x,t)=(T(x),t+\theta \eta (x)), \label{eq.skewproduct}
\end{equation}
where $T(x)$ is a uniformly expanding interval map, $S^1$ is the
circle, $\theta\in[0,1]$ an irrational number, and $\eta (x)$ a
specified ``skewing'' function.  We state the following Theorem, whose
proof, and precise form of $f(x,t)$ is described in
Section~\ref{oscill}.

\begin{thm}
\label{the:thereareoscillations} There are measure preserving systems $ (f,X,\mu )$
of the skew product form \eqref{eq.skewproduct}
which preserve a probability measure $\mu $ and have polynomial decay of correlations.
Moreover for an infinite observable of the kind $\phi(x)=d(x,\tilde{x})^{-k}$,
for some $\tilde{x}\in X$, we have
\begin{equation*}
\liminf_{n\rightarrow \infty }\frac{\log S_{n}(x)}{\log n}
<\limsup_{n\rightarrow \infty }\frac{\log S_{n}(x)}{\log n},
\end{equation*}
for $\mu$-a.e.\ $x\in X$. Furthermore there are measure preserving systems
with polynomial decay of correlations where even the limsup, and power law
behaviour of the Birhkhoff sums does not follow the ratio $\frac{k}{d_{\mu
}(\tilde{x})}$ suggested by \eqref{eq2}, even along subsequences. In such systems
for $\mu$-a.e.\ $x\in X$
\begin{equation*}
\limsup_{n\rightarrow \infty }\frac{\log S_{n}(x)}{\log n}<\frac{k}{d_{\mu
}(\tilde{x})}.
\end{equation*}
\end{thm}

\subsection{Application to extreme events and hitting times in syst\-ems
having a fast mixing return map.}
\label{extremes}

The estimates on Birkhoff sums of infinite observables are useful to
investigate quantitative aspects of the dynamics of systems preserving
an infinite measure. Consider an infinite system $(f,X,\mu)$, where
$\mu$ is assumed to be infinite but $\sigma$-finite. A classical
approach to study such an infinite system is by inducing the
dynamics on a subset $Y$ of positive (finite) measure. Let $R$
be the return time function to the domain $Y$, that is for $x\in Y$,
\begin{equation*}
  R(x) = \min \{\, n > 0 : f^n (x) \in Y \, \}.
\end{equation*}
Then $f^R \colon x \mapsto f^{R(x)} (x)$ defines a dynamical system
$(f^R, Y)$ which preserves the measure $\mu_Y = \mu|_Y$. (We shall now
denote $f^R$ by $f_Y$).  The system $(f_Y,Y,\mu_Y)$ is called the
induced system. It is a finite measure preserving system, and its
dynamics gives information on the original infinite system. The
original system can be seen as a suspension on the induced system,
that is if we define $\hat{f} \colon \hat{Y} \to \hat{Y}$ by
\begin{equation*}
\hat{X} = \{\, (x,n) \in Y \times \mathbb{N} : 0 \leq n < R(x) \,\},
\end{equation*}
and
\begin{equation*}
\hat{f} (x,n) = \left\{
\begin{array}{ll}
(x, n+1) & \text{if } n+1 < R(x), \\
(f_Y (x), 0) & \text{if } n+1 = R(x),%
\end{array}
\right.
\end{equation*}
then $(\hat{f},\hat{X})$ is a suspension of $(f_Y,Y)$ and $(\hat{f}, \hat{X},
\hat{\mu})$ is isomorphic to $(f,X,\mu)$ if $\hat{\mu}$ is defined in the
natural way, e.g.\ see \cite{Zwm}.

Here a major role is played by the return time function $R$. In this
case $R$ will be a non-integrable observable on $(f_Y,Y,\mu_Y)$, and
to this situation we can apply the findings of the previous
section. We remark that the observable $\phi\equiv R$ is not necessarily related to the distance from a certain point. In particular if we are interested in hitting time or extreme
problems then the asymptotic behaviour of the Birkhoff sums $S_{n}^{R}$ of the return
time in the induced system is particularly important. We show
in the following Sections~\ref{s1} and \ref{sec.max-hit} that the behaviour of
$S_n^R$ implies anomalous scaling behaviour for the hitting
time to small targets, and for growth of extreme events.

\subsubsection{Logarithm law and the anomalous hitting time behaviour in
infinite systems.}
\label{s1}
We derive a limit (logarithm) law for the hitting time function.
First, recall some definitions relating to the hitting time to general
targets and logarithm laws in this context. Consider a dynamical
system $(f,X,\mu)$ on a metric space $X.$ Let $B_{n}\subseteq X$ be a
decreasing sequence of targets; let us consider the hitting time of the
orbit starting from $x\in X$ to the target $B_{n}$%
\begin{equation*}
\tau (f,x,B_{n})=\min \{\, n\geq 0 :f^{n}(x)\in B_{n} \,\}
\end{equation*}
(In case the map considered is obvious from the context, instead of $\tau
(f,\cdots)$ we may write $\tau (\cdots)$ for simplicity.)

The classical logarithm law results relates the hitting time scaling
behaviour to the measure of the targets. In many cases when $f$ preserves a
probability measure $\mu $ and the system is chaotic enough, or it has
generic arithmetic properties then the following holds for $\mu$-a.e.~$x\in X$
\begin{equation}
\lim_{n\rightarrow \infty }\frac{\log \tau (f,x,B_{n})}{\log n}
=\lim_{n\rightarrow \infty }\frac{\log (\mu (B_{n}))}{-\log n}.
\label{loglaw}
\end{equation}%
In words: \emph{the hitting time scales as the inverse of the measure of the
targets} (compare with \eqref{ll}).

We see that \emph{in systems preserving an infinite measure} this law
does not hold anymore, but under some chaoticity assumptions we can
replace the equality \eqref{loglaw}, with a rescaled version of
it. In fact, the rescaling factor depends on the return time behaviour of
the system on some subset $Y \supseteq B_{n}$ containing the target
sets $B_{n}.$

Suppose $(f,X,\mu )$ preserves an infinite measure
$\mu$, $Y \subseteq X$ is such that $\mu (Y) <\infty $, and
consider $B_{n}\subseteq Y$. The following holds.

\begin{prop} \label{start}
  Let $(f,X,\mu )$ be a dynamical system preserving an infinite
  measure $\mu $. Let $(f_Y, Y,\mu _{Y})$ be the induced system over a
  domain $Y $ of finite positive measure, preserving a probability
  measure $\mu _{Y}=\mu |_{Y}$, and with return time function $R\colon
  Y\rightarrow \mathbb{N}$. Suppose $R$ has regular suplevels with
  associated exponent $\alpha _{R}$ (see
  Definition~\ref{reg}). Suppose that $(f_Y, Y ,\mu _{Y})$ satisfies
  Condition (SPDC). Let $B_{n}\subseteq Y$ be a decreasing sequence of
  targets also satisfying items (i) and (ii) of
  Definition~\ref{reg}. Consider $\alpha \geq
  0$, such that $\alpha =\lim_{n\rightarrow \infty} \frac{\log \mu
    (B_{n})}{-\log n}.$ Then for $\mu$-a.e.\ $x\in X$,
  \begin{equation*}
    \lim_{n\rightarrow \infty } \frac{\log \tau (f,x,B_{n})}{\log n}
    = \frac{\alpha }{\alpha _{R}}.
  \end{equation*}
\end{prop}

The proof of Proposition~\ref{start} is in
Section~\ref{sec:proofstart}. In the next section achieve a
corresponding statement for the maxima process $M_n$.

\subsection{On the link between almost sure growth of maxima and hitting
time laws.}\label{sec.max-hit}

Suppose $X\subset\mathbb{R}^d$, and consider a sequence of functions
$H_n \colon X\to\mathbb{R}$.  For $x\in X$ consider the following
maximum function sequence and corresponding hitting time function
sequence defined by
\[
\tilde{M}_n (x):= \max_{0 \leq k < n} H_k(x),\quad
\tau_{u}(x)=\min\{\, n \geq 0 : H_n(x)\geq u \,\}.
\]

Examples include the case where we have a probability space
$(X,\mathcal{B}_X,\mu)$, with $\mathcal{B}_X$ the $\sigma$-algebra of
subsets of $X$, $\mu$ a probability measure, and $(H_n)$ a sequence of
random variables. Another case includes that of a measure preserving
system $(f,X,\mu)$, where we set $H_n(x)=\phi(f^n(x))$, with specified
observable function $\phi \colon X\to\mathbb{R}$.  In this latter case,
$\tilde{M}_n(x)$ coincides with the usual definition of $M_n$ given
in \eqref{eq.max-process}.

In this section, we
derive a precise link between the growth rate of $\tilde{M}_n(x)$ (as
$n\to\infty$), and the growth rate of $\tau_u(x)$ (as $u\to\infty$).
We'll assume further that either $\tilde{M}_n(x)\to\infty$ as $n\to\infty$, or
$\tau_u(x)\to\infty$ as $u\to\infty$.

First we make the basic observation that the event
$\{\tilde{M}_n \leq u\}$ is the same as $\{\tau_u\geq n\}$.
We state our first elementary result.

\begin{prop}\label{prop.maxhit}
  Suppose $X\subset\mathbb{R}^d$, and consider the sequence of
  functions $H_n \colon X\to\mathbb{R}$.
  \begin{enumerate}
  \item Suppose that $\ell_1,\ell_2 \colon [0,\infty)\to[0,\infty)$
      are monotone increasing functions, such that
      $\ell_1(u)$, $\ell_2(u)\to\infty$ as $u\to\infty$.  Suppose for
      given $x\in X$, there exists $N(x)>0$, such that for all $n\geq
      N$ we have $\ell_1(n)\leq \tilde{M}_n(x)\leq\ell_2(n)$. Then
      there exists $u_0(x)$, such that for all $u\geq u_0$ we have
      \[
      \ell_2^{-1}(u-1)\leq\tau_u(x)\leq\ell^{-1}_1(u+1).
      \]
    \item Suppose that $\hat{\ell}_1,\hat{\ell}_2 \colon
      [0,\infty)\to[0,\infty)$ are monotone increasing functions, such
          that $\hat{\ell}_1(u)$, $\hat{\ell}_2(u)\to\infty$ as
          $u\to\infty$. Suppose that for given $x\in X$, there exists
          $u_0(x)>0$, such that for all $u\geq u_0$, we have
          $\hat{\ell}_1(u)\leq\tau_u(x)\leq\hat{\ell}_2(u).$ Then
          there exists $N(x)>0$, such that
      \[
      \hat{\ell}^{-1}_2(n)\leq \tilde{M}_n(x)\leq\hat{\ell}^{-1}_1(n).
      \]
  \end{enumerate}
\end{prop}

\begin{rmk}
In the statement of Proposition~\ref{prop.maxhit}, we do not assume
that $X$ is a measure space. In the case where $(H_n)$ is a stationary
process, defined on a suitable measure space, then
Proposition~\ref{prop.maxhit} asserts that almost sure bounds for
$\tilde{M}_n$ imply almost sure bounds for $\tau_u$, and vice
versa. We will use this fact in our dynamical systems applications.
\end{rmk}

Proposition~\ref{prop.maxhit} is proved in Section~\ref{sec.proofmaxlaw}.
We now consider specific applications of this result. For a measure
preserving dynamical system $(f,X,\mu)$ define
\[
\tau^{\phi}_{u}(x)=\inf\{\, n:\phi(f^n(x))\geq u\,\},
\]
and put $\tilde{M}_n(x)=M_n(x)$, where we recall $M_n(x)=\max_{k\leq n-1}\phi(f^{k}(x))$.
Here $\phi \colon X\to\mathbb{R}$ is an observable function.
Examples include $\phi(x)=-\log d(x,\tilde{x})$ or
$\phi(x)=d(x,\tilde{x})^{-1}$, for a given $\tilde{x}\in X$, but we have seen that
our theory allows us to consider much more general cases.

\subsubsection{The logarithm law for hitting times and maxima}
We now consider the \emph{logarithm law}, especially for the hitting
time function.  We show via Proposition~\ref{prop.maxhit} that a
logarithm law for hitting time implies a logarithm law for maxima (and
conversely).  Again this is a pointwise result. See
\cite[Proposition~11]{G4} for a similar statement.

\begin{prop}\label{poi}
  Consider a dynamical system $(f,X)$. Suppose that
  $0<a_1<a_2<\infty$, and $x\in X$. Then we have the following
  implications.
  \begin{align*}
    \limsup_{n\rightarrow \infty } \frac{\log [M_{n}(x)]}{\log n}
    &=a_1 \quad \Longleftrightarrow \quad \liminf_{u\rightarrow \infty
    } \frac{\log [\tau^{\phi}_{u}(x)]}{\log u}=\frac{1}{a_1},
    \\ \liminf_{n\rightarrow \infty } \frac{\log [M_{n}(x)]}{\log n}
    &=a_2 \quad \Longleftrightarrow \quad \limsup_{u\rightarrow \infty
    } \frac{\log [\tau^{\phi}_{u}(x)]}{\log u}=\frac{1}{a_2}.
  \end{align*}
  Moreover, if $a_1=a_2$, then
  \begin{equation*}
    \lim_{u\rightarrow \infty }\frac{\log [\tau^{\phi}_{u}(x)]}{\log
      u}=\left( \lim_{n\rightarrow \infty }\frac{\log [M_{n}(x)]}{\log
      n}\right) ^{-1}
  \end{equation*}
  provided the corresponding limits exist at $x$.
\end{prop}

\begin{rmk} In Proposition~\ref{poi}, the logarithm function diminishes any behaviour associated to sub-polynomial corrections associated to the growth of $M^{\phi}_n$ (as $n\to\infty$), or to that of $\tau^{\phi}_u$ (as $u\to\infty$).
In certain cases, this subpolynomial growth can be further quantified
as we discuss below.
\end{rmk}
Proposition~\ref{poi} is proved in Section~\ref{sec.proofmaxlaw}.
For infinite systems, we state the following corollary concerning the almost sure behaviour
of the maxima process.
\begin{cor}
\label{t2} Let $(f, X,\mu )$ be a dynamical system preserving an infinite
measure $\mu $. Let $(f_Y,Y,\mu_Y)$ and $R$ be as in Proposition~\ref{start}. Consider $\phi \colon X\rightarrow [0,\infty)$, a function which is
not bounded and that $\phi |_{X \setminus Y}$ is bounded. Suppose that also $%
\phi$ has regular suplevels. Then
\begin{equation*}
\lim_{n\rightarrow \infty }\frac{\log [M_{n}(x)]}{\log n}=\frac{\alpha _{R}}{
\alpha _{\phi }}
\end{equation*}
holds for $\mu $-a.e.~$x\in X.$
\end{cor}
The proof of Corollary~\ref{t2} can be found in Section~\ref{sec.proofmaxlaw}.
In Corollary~\ref{t2}, we have assumed (SPDC) for $(f_Y,Y,\mu_Y)$. In
the case where $(f_Y,Y,\mu_Y)$ satisfies stronger hypotheses, such as
being Gibbs--Markov, then we can obtain stronger bounds on almost
sure behaviour of the maxima function (as $n\to\infty$), and also
the hitting time function via Proposition~\ref{prop.maxhit}. We remark further that
Corollary~\ref{t2} gives almost sure bounds on the maxima process in the case of observables having general
geometries (beyond functions of distance to a distinguished point).
Thus if we know (almost sure) bounds on the hitting time function,
then we get corresponding bounds for the maxima process via Proposition~\ref{prop.maxhit} (or \ref{poi}).
This result allows us to address a question posed in e.g.\ \cite[Section~6]{HNT}
concerning the existence of an almost sure behaviour of maxima for general
observables (that are not solely a function of distance to a distinguished point).

\subsubsection{On finding precise asymptotics on the maxima and hitting time functions}
For certain stationary processes (or dynamical systems), the rate functions $\ell_1(n),\ell_2(n)$ as appearing in Proposition~\ref{prop.maxhit} can be optimised. For i.i.d.\ processes $(H_n)$, optimal expressions for these functions are given in
e.g.\ \cite{Embrechts, Galambos}. For dynamical systems having exponential decay of
correlations, higher order corrections to the almost sure
maxima function growth (beyond that given by a standard logarithm law in Proposition~\ref{poi}) are discussed in e.g.\  \cite{GNO,HNT}.
To translate such results to almost sure behaviour of hitting times, then inversion of the functions $\ell_1(n),\ell_2(n)$ is required. This we now discuss via an explicit example.
Generalisations just depend on an analysis of the functional forms of $\ell_1(n)$
and $\ell_2(n)$.

Consider the \emph{tent map} $T_{2}(x)=1-|2x-1|$, $x\in[0,1]$, and the
observable function $\phi(x)=-\log d(x,\tilde{x})$. It is shown that
there exist explicit constants $c_1, c_2>0$ such that for
Lebesgue-a.e.\ $\tilde{x}\in[0,1]$
\[
\log n-c_1\log\log n\leq M_n(x)\leq \log n+c_2\log\log n,
\]
eventually in $n$ for Lebesgue-a.e.\ $x\in [0,1]$, see \cite{HNT}. We
deduce the following asymptotic for the hitting time function. A proof
is given in Section~\ref{sec.proofmaxlaw}.

\begin{lemma}\label{lem.tent}
  Consider the tent map $T_{2} \colon[0,1]\to[0,1]$, and observable
  $\phi(x)=-\log d(x,\tilde{x})$. Then for all $\tilde{x}\in[0,1]$,
  and $\mu$-a.e.\ $x\in[0,1]$, there exists $u_0>0$ such that for all
  $u\geq u_0$
  \[
  \log\tau^{\phi}_{u}(x)=u +O\left(\log u\right).
  \]
  Here the $O(\cdot)$ constant depends on $x\in[0,1]$, and on
  $c_1,c_2$.
\end{lemma}

\begin{rmk}
  We immediately deduce a logarithm law for entrance to balls
  $B(\tilde{x},r)$ with an error rate. In particular if we let
  $u=(-\log r)^{-1}$, then $\tau^{\phi}_{u(r)}(x,\tilde{x})=\inf\{\, n
  : d(x,\tilde{x})\leq r\,\}$. Then by Lemma~\ref{lem.tent}, we obtain
  for $\mu$-a.e.\ $x\in X$ that
  \[
  \log\tau^{\phi}_{u(r)}(x)=-\log r +O\left(\log \log r\right).
  \]
\end{rmk}

In the example above, we used a higher order asymptotic on the growth
rate for the maxima to deduce a similar asymptotic for the hitting
time function. We note that a converse result applies if we have
knowledge of such asymptotics for the almost sure growth of the
hitting time function, but no apriori bounds for the growth of the
maxima function. For either the hitting time function, or maxima
function almost sure growth rates are usually deduced via
Borel--Cantelli arguments, see \cite{Galambos, GNO, HNT} for maxima,
and \cite{GK} for hitting times. We elaborate in
Section~\ref{sec.bcmaxhit}.  Thus, Proposition~\ref{prop.maxhit}
allows us to translate limit laws between maxima and hitting times
without too much extra work, except for estimating inverses of the
corresponding rate functions. We remark that in the case of
distributional limits for maxima and hitting times (as opposed to
almost sure bounds), a relation between their limit laws is described
in \cite{FFT1, FFT2}.

\subsection{Dynamical Borel--Cantelli Lemmas for infinite measure preserving
systems}\label{BC}

For a (probability) measure preserving dynamical system $(f,X,\mu
)$, a dynamical Borel--Cantelli Lemma result asserts that for a
sequence of sets $(B_{n})$ with $\sum_{n}\mu (B_{n})=\infty $, we have
\begin{equation*}
\mu \left( \bigcap_{i=1}^{\infty }\bigcup_{n=i}^{\infty }\{\,x:f^{n}(x)\in
B_{n}\,\}\right) =1,
\end{equation*}
i.e.\ $\mu \{\,x\in X:f^{n}(x)\in B_{n},\text{infinitely often}\,\}=1$.
A quantitative version leads to having the \emph{strong Borel--Cantelli}
(SBC) property defined as follows.
Given a sequence of sets $(B_{n})$ with $\sum_{n}\mu (B_{n})=\infty $, let $%
E_{n}=\sum_{k=0}^{n-1}\mu (B_{k})$.

\begin{defn}
We say that $(B_n)$ satisfies the strong Borel--Cantelli property
(SBC) if for $\mu$-a.e.\ $x\in X$
 \begin{equation*}
    \lim_{n\to\infty}\frac{S_n(x)}{E_n}=1,
 \end{equation*}
 where $S_n (x) = \sum_{k=0}^{n-1} 1_{B_k} (f^{k}x)$, and $1_{B_k}(x)$ denotes the indicator
function on the set $B_k$.
\end{defn}

For dynamical systems preserving a probability measure $\mu$, (SBC) results are now
known to hold for various systems, see \cite{GNO, HNPV, HNT,Kim}. Here, we
derive corresponding Borel--Cantelli results for infinite systems
$(f,X,\mu)$,
with $\mu$ a $\sigma$-finite measure, and $\mu(X)=\infty$.

We consider a conservative, ergodic system $(f,X,\mu)$, and suppose there
exists $Y\subset X$ for which the induced system $(f_{Y},Y,\mu _{Y})$ is
Gibbs--Markov (see Sections~\ref{sec:hittingtime}, \ref{extremes} for
conventions), but now the return time function $R\colon Y\rightarrow Y$ is
not integrable with respect to $\mu _{Y}$. In the case of integrable return
times, \cite[Theorem~3.1]{Kim} established strong Borel--Cantelli results
for the system $(f,X,\mu )$ assuming strong Borel--Cantelli results for the
induced system $(f_{Y},Y,\mu _{Y})$. Formally, consider a function sequence $%
p_{j}$ with $\sum_{j}\mu (p_{j})=\infty $, where $\mu (p_{j})=\int
p_{j}(x)\,d\mu $. We say that the strong Borel--Cantelli property holds for
this sequence, with respect to $(f_{Y},Y,\mu _{Y})$, if (necessarily)
$\sum_{j=1}^{n}\mu(p_{Q(j,x)})\rightarrow \infty $ as $n\rightarrow \infty $ and
\begin{equation}
\frac{\sum_{j=1}^{n}p_{Q(j,x)}(f_{Y}^{j}(x))}{\sum_{j=1}^{n}\mu (p_{Q(j,x)})}
\rightarrow 1,
\end{equation}
where $Q(j,x)=\sum_{i=0}^{j}R(f_{Y}^{i}(x))$ is the total clock time
associated to the $j$'th return to the base. We now
state the corresponding dynamical Borel--Cantelli result as applicable for
infinite systems.

\begin{thm} \label{thm:infiniteSBC}
  Suppose that $(f,X,\mu)$ is ergodic and conservative, and that the
  induced system $(f_Y,Y,\mu_Y)$ satisfies
  assumptions~\ref{G-M1}--\ref{G-M6}, (where observable $\phi(x)$ is
  identified with $R(x)$) and $\beta$ given by \ref{G-M6}. Put $\alpha
  = 1/\beta$. Let $(p_n)$ be a sequence of non-negative funtions which
  satisfy $p_1\geq p_2 \geq \ldots$, and assume further that
  $\mathrm{supp}(p_n)\subset Y$. We have the following cases.

  \begin{enumerate}
  \item Suppose that there exists $\varepsilon_1 \in (0, \alpha)$ such
    that $\sum_{n\geq 1}\mu(p_{n^{\alpha+\varepsilon_1}})=\infty$. If
    every subsequence $p_{n_k}$ with $\sum\mu(p_{n_k})=\infty$ is a
    strong Borel--Cantelli sequence with respect to $f_Y(x)$, then we
    have for all $\varepsilon \in (0, \varepsilon_1]$, and eventually
    as $n\to\infty$, that
    \begin{equation}  \label{eq.quant}
      \sum_{k = 1}^{n^{\frac{1}{\alpha + \varepsilon}}} \mu
      (p_{k^{\alpha+\varepsilon}}) \leq \sum_{k=1}^{n} p_k (f^k (x))
      \leq \sum_{k=1}^{n^{\frac{1}{\alpha-\varepsilon}} }
      \mu(p_{k^{\alpha-\varepsilon}}),
    \end{equation}
    for $\mu$-a.e.\ $x\in X$.

  \item Suppose there exists $\varepsilon_2>0$ such that $\sum_{n\geq
    1}\mu(p_{n^{\alpha-\varepsilon_2}})<\infty$, then
    \begin{equation*}  
      \lim_{n\to\infty}\sum_{k=1}^{n} p_k (f^k(x))<\infty,\qquad
      \text{for } \mu\text{-a.e. } x\in X.
    \end{equation*}
  \end{enumerate}
\end{thm}
(In the statement of the theorem, $p_{g(n)}$ should be interpreted as
$p_{[ g(n) ]}$, where $[ \cdot ]$ denotes the integer part.)

We make several remarks and discuss immediate consequences of
Theorem~\ref{thm:infiniteSBC}. Firstly, with slightly more effort, it
is possible in item (1) to replace the correction by $\pm \varepsilon$
in the exponents with corrections by logarithms. In some cases, if $Y$
is a Darling--Kac set, then we can get an even more precise upper
bound, see Aaronson--Denker \cite{AaronsonDenker} and the proof of
Theorem~\ref{thm:infiniteSBC} for more details.

A condition of Theorem~\ref{thm:infiniteSBC} is that we must assume
that
\[
\sum_{n\geq 1}\mu(p_{n^{\alpha+\varepsilon_1}})=\infty
\]
for some $\varepsilon_1>0$. This puts a restriction on the sequence of
functions $p_n$. Indeed it is not difficult to construct sequences
with $\sum_{n\geq 1}\mu(p_{n})=\infty$, but $\sum_{n\geq
  1}\mu(p_{n^{\alpha+\varepsilon}})<\infty$. If for example
$\mu(p_n)=n^{-\zeta}$, ($\zeta>0$), then we require $\zeta<\beta(1+
\varepsilon\beta)^{-1}$. In the case $\beta \to 1$, we find that any
$\zeta<1$ will do. In the case of item (2), if
\[
\sum_{n\geq 1} \mu (p_{n^{\alpha-\varepsilon_2}}) < \infty,
\]
(for some $\varepsilon_2>0$), then via an argument using the first
Borel--Cantelli Lemma we show that for $\mu$-a.e.\ $x\in X$ we have
\begin{equation*}
  \sum_{k=1}^{\infty} p_k (f^k(x))<\infty.
\end{equation*}
We remark that the bounds given in equation \eqref{eq.quant}
appear mysterious at first glance. In the case where $\mu(p_n)$ is
described by a functional sequence $g(n) = \mu (p_n)$, with $g
\colon (0,\infty) \to (0,\infty)$ a monotone decreasing real valued
function, then a simple change of variable argument implies that the
bounds in equation \eqref{eq.quant} can be written as
\begin{equation*}
\sum_{k=1}^{n^{\frac{1}{b}} }
\mu(p_{k^{b}})\sim\sum_{k=1}^{n} k^{1 - \frac{1}{b}} g(k),
\end{equation*}
with $b=\alpha\pm\varepsilon$ accordingly. Furthermore, in the case $\mu$ is
a probability measure then, (as established in \cite{Kim}) the usual strong
Borel--Cantelli property holds:
\begin{equation*}
  \lim_{n\to\infty}\frac{\sum_{k=1}^{n} p_k (f^k
    (x))}{\sum_{k=1}^{n}\mu(p_{k}) } = 1.
\end{equation*}
In this case, the $\varepsilon$ in Theorem~\ref{thm:infiniteSBC} is
not needed. The boundary case arises when $\beta=1$ (and $\int R
\,d\mu=\infty$). Here we can take $\alpha=1$, but the $\varepsilon$ is
still required in equation \eqref{eq.quant}. As in \cite{Kim}, the
Gibbs--Markov assumption is then not required in the case $\mu$ is a
probability measure. We require the Gibbs--Markov assumption to get
quantitative (lower) bounds on the maximum growth of the return time
function, see Lemma~\ref{slem.gamma}.

In the case where $f_Y$ is a Gibbs--Markov map, or non-uniformly expanding
map with fast decay of correlations, the sequence of functions $(p_j)$ that lead to
the strong Borel--Cantelli property include indicator functions of balls.
More general classes of functions may also lead to the strong
Borel--Cantelli property, see e.g.\ \cite{Gouezel,GNO, Kim}.

A further question that arises is what can be said about dynamical
Borel--Cantelli results for functions $(p_{n})$ no longer supported on $Y$?
In general, Theorem~\ref{thm:infiniteSBC} gives no immediate answer. The
fact that the return time function $R$ is a \emph{first return time} allows us
to track the frequency of visits of typical orbits to the regions where $%
p_{j}>0$. If these functions are not supported on $Y$, then an orbit can
have multiple visits to the regions where $p_{j}>0$ before returning to $Y$,
and this visit frequency cannot in general be controlled. However, as we
will see for a family of intermittent maps described in Section~\ref{sec.intermittent},
it is sometimes possible to explicitly track orbits once they leave $Y$, and
hence establish dynamical Borel--Cantelli results for functions $(p_{n})$
that are no longer assumed to be supported on $Y$.

As a further application, we also establish dynamical Borel--Cantelli
results for infinite systems that are modelled by Young towers,
\cite{Young}. This is discussed in Section~\ref{sec:youngtowers}. This
builds upon the work of \cite{GNO}, where they establish dynamical
Borel--Cantelli results for Young towers in the case of the system
preserving a probability measure.

\subsection{Refined limit laws for maxima and quantitative results on hitting time laws for the Markov case}
\label{sec.bcmaxhit}

We have seen in Section~\ref{sec.max-hit}, via Proposition~\ref{poi}, that
a logarithm law for maxima $M_n(x)$ can be achieved if a logarithm law for
the hitting time function $\tau^{\phi}_u(x)$ is known, especially
for systems $(f,X,\mu)$ built over an induced system $(f_Y,Y,\mu_Y)$
satisfying Condition (SPDC). In this section we establish refined
growth rates of maxima $M_{n}$ for infinite systems via knowledge of a
strong Borel--Cantelli result for the induced system $(f_Y,Y,\mu_Y)$.
Note, via Proposition~\ref{poi} we achieve almost sure
bounds for $\tau^{\phi}_u(x)$.
To keep the exposition simple we assume a Gibbs--Markov property for
$(f_Y,Y,\mu_Y)$. We point out various generalisations below.

We suppose that $(f, X,\mu)$ is an ergodic, conservative, and preserving a
$\sigma$-finite (infinite) measure $\mu$.  Let $\psi$ be a
monotonically decreasing measureable function, and let
$\phi(x)=\psi(d(x,\tilde{x}))$. For the induced system $(f_Y,Y,\mu_Y)$ we
consider the case where $\tilde{x}\in Y$, so that for sufficiently
large $u$, the set $\{\, x :\phi(x)\geq u \,\}$ is supported in
$Y$. In specific applications, we show that this constraint can be
sometimes relaxed. We state the following result. As before,
$\alpha=1/\beta$.

\begin{thm} \label{thm:infinitemaxgibbs}
  Suppose that $(f, X,\mu)$ is conservative and ergodic, and that the
  induced system $(f_Y,Y,\mu_Y)$ with return-time $R \colon Y\to \mathbb{N}$
  satisfies assumptions~\ref{G-M1}--\ref{G-M6} (with $R$ playing the
  role as the observable in \ref{G-M6}). For a monotonically
  decreasing function $\psi \colon [0,\infty)\to[0,\infty)$, assume
      that $\phi(x)=\psi(d(x,\tilde{x}))$ for some point $\tilde{x}\in
      Y$, and assume the density of $\mu$ exists and is positive at
      $\tilde{x}$.  Then for all $\varepsilon>0$, and $\mu$-a.e.~$x\in X$, there exists an $N_x$ such that for all $n\geq
      N_x$
  \begin{equation}\label{eq.maxquant}
    \psi \biggl( \frac{(\log n)^{c_1}}{n^{1/\alpha}}
    \biggr) \leq M_n(x) \leq \psi \biggl( \frac{1}{n^{1/\alpha} (\log
      n)^{c_2}} \biggr),
  \end{equation}
for some constants $c_1,c_2>0$.
\end{thm}
Using Proposition~\ref{prop.maxhit} we can then obtain bounds
on the almost sure behaviour of $\tau^{\phi}_u(x)$ (as $u\to\infty$).
For certain forms of $\psi$ these bounds
can be made explicit. We state the following,
whose proof is similar to that of Lemma~\ref{lem.tent}.
\begin{cor} \label{cor.hit-int}
  Let $(f,X,\mu)$, and $\tilde{x}$ be as in
  Theorem~\ref{thm:infinitemaxgibbs}, and let $\phi(x)=-\log
  d(x,\tilde{x})$.  Then for $\mu$-a.e.\ $x\in[0,1]$, there exists
  $u_0>0$ such that for all $u\geq u_0$,
  \[
  \log\tau^{\phi}_{u}(x)=\alpha u +O\left(\log u\right).
  \]
  Here the $O(\cdot)$ constant depends on $x\in[0,1]$, and on
  $c_1,c_2$ as appearing in Theorem~\ref{thm:infinitemaxgibbs}.
\end{cor}

We make the following remarks. The proof of
Theorem~\ref{thm:infinitemaxgibbs} uses explicit almost sure bounds
achieved on the maxima process for the Gibbs--Markov map
$f_Y$, see e.g.\ \cite[Proposition~3.4]{HNT} which utilises a
\emph{quantitative strong Borel--Cantelli} (QSBC) property for the
system $(f_Y,Y,\mu_Y)$. Informally this property is described as follows: if
$p_k$ is a decreasing sequence of functions, with
$E_n:=\sum_{k=1}^{n}\mu_Y(p_k)\to\infty$, then a QSBC property takes
the form
\begin{equation*} 
  \sum_{k=1}^{n}p_k(f_{Y}^k(x))=E_n+O\left(E_{n}^{\beta'}\right),\qquad
  \text{for } \mu_Y\text{-a.e.}~x\in Y,
\end{equation*}
and for some $\beta'\in (0,1)$. To get the logarithmic correction terms in
equation \eqref{eq.maxquant}, the QSBC property is used. If instead a
\emph{standard} SBC property is used, i.e.\ ignoring the
$E_{n}^{\beta'}$ correction, then we obtain slightly weaker estimates
on the bounds for the maxima as given in
Theorem~\ref{thm:infinitemaxgibbs}. These latter bounds are still
sufficient to obtain a logarithm law for the entrance time via
Proposition~\ref{prop.maxhit}, and commensurate with
Corollary~\ref{t2}. The results are in fact consistent with the
approaches used in \cite{GK}, where a logarithm law of the hitting
time function is obtained using (standard) strong Borel--Cantelli
assumptions.  Notice further, that we could have worked directly with
the system $(f,X,\mu)$, and the Borel--Cantelli property achieved in
Theorem~\ref{thm:infiniteSBC} to deduce results on almost sure growth
rates of maxima and hitting times.  However if we had done this, we
would have lost information in the (logarithmic) asymptotic
corrections and obtained suboptimal results.

\subsection{Applications: Hitting times, extremes, and run length for intermittent maps with a $\mathbf{\sigma}$-finite measure.}
\label{sec.intermittent}
Let $S^{1}=[0,1)$, and we finally consider application of our results of
Sections~\ref{sec:mainresults}
to the following family of intermittent maps
$(f_{\alpha},S^{1},\mu)$, where $f_{\alpha} \colon S^{1} \to S^{1}$ is defined by
\begin{equation}\label{eq.LSVmap}
f_{\alpha} (x) = \left\{ \begin{array}{ll} x ( 1 + 2^\alpha x^\alpha )
  & x < 1/2, \\ 2x - 1 & x \geq 1/2. \end{array} \right.
\end{equation}
See Figure~\ref{fig.mp}.
\begin{figure}
  \includegraphics{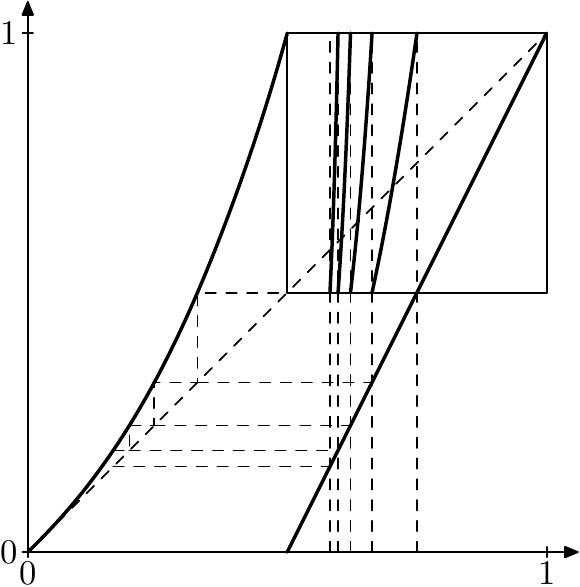}
  \caption{An intermittent map, with induced map in top right quadrant.}\label{fig.mp}
\end{figure}
In the dynamical systems literature, these maps have been well
studied, e.g.\ for their mixing properties \cite{Gouezel, LSV,
  Melbourne-Terhesiu, Young, Young2}, and also their recurrence properties in
relation to entrance time statistics, extremes and dynamical
Borel--Cantelli results \cite{Aar2,Gouezel,HNT,Kim,V.et.al}, to name a
few. We shall focus on the case $\alpha\geq 1$, for which the map
preserves a $\sigma$-finite measure $\mu$, with $\mu(X)=\infty$
\cite{Aar2}.  This measure is absolutely continuous with respect to the
Lebesgue measure, but has a non-normalisable density function. One
ergodic, invariant probability measure which is physically meaningful for this map is the Dirac measure at $\{0\}$, and
thus at first glance the (long-run) statistics appear to be
trivial. However, the orbit of Lebesgue almost-every $x\in S^{1}$ is
dense, and thus it is natural to study the asymptotic recurrence
behaviour of typical points captured by the statistics of the infinite
measure $\mu$. This makes the problem of analysing the statistics of
entrance times, extremes and dynamical Borel--Cantelli results an
interesting one.

To apply the results already established in the earlier part of
Section~\ref{sec:mainresults}, we consider an induced
system $(f_Y,Y,\mu_Y)$ (we drop the subscript $\alpha$) together with
a first return time function $R \colon Y\to\mathbb{N}$. We take
$Y=(1/2,1]$, and hence take return-time $R \colon Y\to\mathbb{N}$ defined by
\[
R (x) = \min \{\, n \geq 1 : f^n (x) \in Y\,\},
\]
with $x\in Y$. As before, we write
\[
f_Y(x) = f^{R(x)} (x), \qquad R_j (x) = R (f_{Y}^j (x)).
\]
We summarise key properties of $f$ and $f_Y$ as follows.
Define the sequence $(x_n)$ by
\[
x_{-1} = 1, \qquad x_0 = 1/2, \qquad f(x_{n+1}) = x_n,
\]
keeping $x_n<1/2$ for all $n\geq 1$. Then we have the following
asymptotic relation,
\[
x_n\sim (\alpha n)^{-1/\alpha}.
\]
The map $f_Y$ is uniformly expanding, and moreover there is a
countable Markov partition $\mathcal{P}=\{\, Y_i : i\in\mathbb{N}
\,\}$, with $R|_{Y_i}$ constant, and $f_Y(Y_i)=Y$. To be more
explicit, for $n\geq 1$, let $z_n \in [1/2,1]$ be such that
$f(z_n)=x_{n-1}$, and let $z_0=1$. If we write $Y_n=[z_n,z_{n-1}]$,
Then $R|_{Y_n}=n$. If
$\mathcal{P}_n=\bigvee_{i=1}^{n}f_Y^{-i}\mathcal{P}$, then for all
$k\leq n$, the iterate $f_Y^k$ satisfies uniform bounded distortion
estimates on all $\omega\in\mathcal{P}_n$. In particular there is a
$f_Y$-invariant probability measure $\mu_Y$ on $Y$, which is
  equivalent to Lebesgue measure. Thus, this map satisfies
  \ref{G-M1}--\ref{G-M6}, and the function $R$ is a first return time
  to $Y$. In particular the return time $R \colon Y\rightarrow
    \mathbb{R}$ has a behaviour
\begin{equation*}
R\left(\frac{1}{2}+x\right)=c(x)x^{-\alpha }
\end{equation*}
with $m\leq c(x)\leq M.$ For $\alpha \geq 1$, $R$ plays the role of an
infinite observable, and hence our theoretical results on maxima growth, hitting time
laws, and Borel--Cantelli results can be
applied. In the context of this example, we also state the results for Lebesgue
measure, rather than the infinite measure $\mu$.  On compact subsets
in $S^{1}\backslash\{0\}$, the restriction of $\mu$ is equivalent to Lebesgue measure.

\subsubsection{Results on extremes and entrance time laws.}
We first study the almost sure growth rate of the maximum process
\[
M_n (x) = \max \{\, \psi (d (f^j x, \tilde{x})) : 0 \leq j < n \,\},
\]
where $\psi \colon [0,\infty)\to\mathbb{R}$ is monotone decreasing
  function, taking its maximum at $0$, and $\tilde{x}\in X$ is given.

\begin{thm} \label{thm:maximaint}
  Suppose $(f_{\alpha},S^{1},\mu)$ is an intermittent map as
  defined in equation \eqref{eq.LSVmap} for $\alpha\geq 1$. Consider
  the observable function $\phi(x)=\psi(d(x,\tilde{x}))$, where $\psi
  \colon [0,\infty) \to \mathbb{R}$ is a monotonically decreasing
  function. Then for all $\varepsilon>0$ and Lebesgue almost all $x
  \in ~S^{1}$, we have the following cases.
  \begin{enumerate}
  \item If $\tilde{x}=0$, then
    \[
    \psi \biggl( \frac{(\log n)^{2 + \varepsilon}}{n^{1/\alpha}}
    \biggr) \leq M_n(x) \leq \psi \biggl( \frac{1}{n^{1/\alpha} (\log
      n)^{2 + \varepsilon}} \biggr),
    \]
    when $n$ is large enough.
  \item If $\tilde{x}\in(0,1]$ we have
    \[
    \psi \biggl( \frac{(\log n)^{4 + \varepsilon}}{n^{1/\alpha}}
    \biggr) \leq M_n(x)\leq \psi \biggl( \frac{1}{n^{1/\alpha} (\log
      n)^{2 + \varepsilon}} \biggr)
    \]
    when $n$ is large enough.
  \end{enumerate}
\end{thm}

We also obtain the corresponding law for the entrance time.
\begin{cor}\label{cor.logithmnlaw}
  For the intermittent maps given in equation \eqref{eq.LSVmap},
let $\tilde{x}\in (\frac{1}{2},1]$.
    Then the hitting time behaviour in balls around $\tilde{x}$ scales as
  \begin{equation*}
    \lim_{r\rightarrow 0 }\frac{\log [\tau_{r}(x,\tilde{x})]}{-\log r}={\max
      (1,\alpha )},
  \end{equation*}
for Lebesgue-a.e.\ $x\in S^1$.
\end{cor}

\begin{rmk}\label{rmk.logithmnlaw}
We remark that for the intermittent maps $(f_{\alpha},S^{1},\mu)$,
even when the \emph{natural} invariant measure is infinite, it is
still absolutely continuous with respect to Lebesgue measure, with
local dimension 1 (except at the origin). In this case the
the exponent ${\alpha}$ plays the role of a rescaling factor
to get a logarithm law for this case.
\end{rmk}

As a comparison, we now consider the growth of maxima in the finite measure case.
Using the inducing technique, we can also apply the previous arguments
to the family $(f_{\alpha},S^{1},\mu)$ in the case where $\alpha<1$.
This allows us to improve on the result
stated in \cite[Corollary~4.2]{HNT}, which is primarily based on
dynamical Borel--Cantelli estimates for systems with polynomial
decay of correlations. For $\alpha < 1$, $\mu$ is now
a probability measure. Hence for almost every $x \in Y$, there is
  a $C = C(x)$ such that:
\[
n \leq \sum_{j = 0}^n R_j (x) \leq C n.
\]
That is, we have the asymptotic $S^{R}_n(x)\in [n,Cn]$, $\mu$-a.e. We therefore
have the following, and the proof follows step by step the arguments above.

\begin{cor}
  Suppose $(f_{\alpha},S^{1},\mu)$ is the map given in
  \eqref{eq.LSVmap}, and defined for $\alpha< 1$. Consider the
  observable function $\phi(x)=\psi(d(x,\tilde{x}))$, where $\psi
  \colon [0, \infty) \to \mathbb{R}$ is a monotonically decreasing
  function, and $\tilde{x}\in S^{1}$. Then for all $\varepsilon>0$ and
  Lebesgue almost every $x \in S^{1}$, we have
  \[
  \psi \biggl( \frac{(\log n)^{4 + \varepsilon}}{n} \biggr) \leq
  M_n(x) \leq \psi \biggl( \frac{1}{n(\log n)^{2 + \varepsilon}}
  \biggr),
  \]
  for all sufficiently large $N$.
\end{cor}

For example, in the case where $\phi(x)=-\log d(x,\tilde{x})$, we have
\[
\lim_{n\to\infty}\frac{M_n(x)}{\log n}=1.
\]
Notice that in the case where $\tilde{x}=0$, the local dimension $d_{\mu}(\tilde{x})$ is
$1-\alpha$, and so we get an anomaly in the growth of $M_n$ at this point.
However, for systems with superpolynomial decay of correlations we
generally expect $\frac{M_n(x)}{\log n}$ to converge to
$1/d_{\mu}(\tilde{x})$.

\subsubsection{Dynamical Borel--Cantelli results}
For the family of intermittent maps $(f_{\alpha},S^{1},\mu)$, we can
apply the techniques and results of Section~\ref{BC} to study
dynamical Borel--Cantelli results for shrinking targets. However,
using the dynamical features of these maps we can extend such results
off the inducing set $Y=(1/2,1]$.

\begin{cor}\label{cor.BC}
Consider the intermittent maps $(f_{\alpha},S^{1},\mu)$ given in equation \eqref{eq.LSVmap},
for $\alpha\geq 1$.  Suppose $(B_n)$ is a decreasing sequence of balls, with
  \[
  \sum_{n\geq 1} \mu (B_{n^{\alpha+\varepsilon_1}}) =
  \infty
  \]
  for some $\varepsilon_1\in(0,\alpha)$, and
  $\{0\}\not\in\cap_n\overline{B}_n$. Then for $\mu$-almost every $x$,
and for all $\epsilon\in(0,\varepsilon_1]$
  \begin{equation*} 
    \sum_{k = 1}^{n^{\frac{1}{\alpha + \varepsilon}}}
    \mu(B_{k^{\alpha+\varepsilon}}) \leq \sum_{k=1}^{n} 1_{B_k}
    (f^k(x)) \leq \sum_{k=1}^{n^{\frac{1}{\alpha-\varepsilon}}
    }\mu(B_{k^{\alpha-\varepsilon}})
  \end{equation*}
  holds eventually as $n\to\infty$.
\end{cor}
We remark that Corollary~\ref{cor.BC} is stated in a basic form so as
to highlight the applicability of our results to the family of
intermittent maps $f_\alpha$.  It is clear generalisations are
possible.
\begin{proof}
  Since $\{0\}\not\in\cap_n\overline{B}_n$, there exists $y>0$ and $n_0>0$ with
  $[0,y)\cap\left(\cap_{n>n_0}B_n\right)=\emptyset$.  Following the
    proof of Theorem~\ref{thm:maximaint}, we can construct a first return map
    $f_{Y}$ over an inducing set $Y$ with $[0,y)\cap
      Y=\emptyset$. This map will satisfy \ref{G-M1}--\ref{G-M6}. The
      remainder of the proof now follows step by step the proof of
      Theorem~\ref{thm:infiniteSBC}, as applied to the induced map.
\end{proof}

\subsubsection{Dynamical run length function and Erd\H{o}s--R\'{e}nyi law}\label{sec.runs}
In this section, we establish dynamical run length results for the
family of intermittent maps $(f_{\alpha},S^{1},\mu)$ in the case
$\alpha\geq 1$, i.e.\ for systems that admit a $\sigma$-finite
(infinite) invariant measure.  Here, we choose the natural partition
$Y^{(c)}=[0,1/2)$ and $Y=[1/2,1)$, and the run length functions $\xi_{n}^{(0)}$
and $\xi^{(1)}_{n}$ as specified in equation
\eqref{def_runlengthfunction} are defined accordingly to this
partition. We state the following result.

\begin{thm}\label{thm_runlength}
Suppose $(f_{\alpha},S^{1},\mu)$ is an intermittent map as
  defined in equation \eqref{eq.LSVmap} for $\alpha\geq 1$.
For Lebesgue almost every $x\in~S^{1}$,
    we have
  \begin{equation}\label{equ_runlength1}
    \lim_{n\to\infty}\frac{\xi^{(1)}_{n}(x)}{\log_{2} n}=\frac{1}{\alpha},
  \end{equation}
  \begin{equation}\label{equ_runlength0}
    \lim_{n\to\infty}\frac{\log\xi^{(0)}_{n}(x)}{\log n}=1.
  \end{equation}
\end{thm}
This result is proved in Section~\ref{sec:runlength}.
It is interesting to note that the typical growth rate of $\xi_n^{(1)}$
depends on $\alpha$, while that of $\xi_n^{(0)}$ does not.
This is in contrast to the corresponding run length results in
the probabilistic cases \cite[Theorem~1]{CFZ18}, due to the additional scaling contribution
arising from the asymptotics of the return time function $R$.

Initialized by Erd\H{o}s--R\'{e}nyi's work\cite{Eroren70}, it is worth
to mention the dynamical run length function is connected to the
Erd\H{o}s--R\'{e}nyi strong law of large numbers. This relates to the
possible limits of the function
\begin{align*}
  \Upsilon(\varphi(x),n, K(n)):= & \max_{0\leq i\leq
    n-K(n)}\left\{S_{i+K(n)}(\varphi)(x)-S_{i}(\varphi)(x)\right\}\\ =&\max\left\{\,
  S_{K(n)}(\varphi)\circ T^{i}(x):0\leq i\leq n-K(n) \,\right\},
\end{align*}
as $n\to\infty$, for prescribed (window) function $K(n)$, and typical $x$.

Based on Theorem~\ref{thm_runlength}, we can easily obtain the
following corollary on Erd\H{o}s--R\'{e}nyi strong law for the particular
case of a characteristic function observable, and window length.

\begin{cor}\label{cor:runlength}
For the intermittent maps given in equation \eqref{eq.LSVmap}, we have
the following.
\begin{itemize}
\item[(1)] For every integer sequence $K(n)$ with
  \[
  \limsup_{n\to\infty}\frac{\alpha K(n)}{\log_{2}n}<1,
  \]
  we have for Lebesgue almost every $x\in S^{1}$
  \[
  \lim_{n\to\infty}\frac{\Upsilon(1_{Y}(x),n,K(n))}{K(n)}=1;
  \]
  \item[(2)] For every integer sequence $K(n)$ with $\limsup_{n\to\infty}\frac{\log K(n)}{\log n}<1$, we have for Lebesgue almost every $x\in S^{1}$.
  \[
  \lim_{n\to\infty}\frac{\Upsilon(1_{Y^{(c)}}(x),n,K(n))}{K(n)}=1.  \]
\end{itemize}
\end{cor}

\begin{proof}
  Without loss of generality, we only prove item $(1)$ in
  Corollary~\ref{cor:runlength}. Since $K(n)\leq
  \frac{\log_{2}n}{\alpha}$, Theorem~\ref{thm_runlength} yields
  that there is at least one $n'<n-K(n)$, such that
  $\varepsilon_{n'+1}=\cdots=\varepsilon_{n'+K(n)}=1$ (as
  $n\to\infty$, and Lebesgue almost surely). Therefore, we have
  $\lim_{n\to\infty}\frac{\Upsilon(1_{Y}(x),n,K(n))}{K(n)}=1$,
  Lebesgue almost surely, as was to be proved.
\end{proof}

\section{Proof of Theorem~\protect\ref{thm1}.}
\label{sec:proofsuperpolynomial}

In this section we prove Theorem~\ref{thm1}. As shown in
Remark~\ref{aarmk} the upper bound for $S_n(x)$ as stated in Theorem~\ref{thm1} can be
recovered from Proposition~\ref{aa}.  To get estimates from below we
begin with the following proposition. Recall that condition (SPDC) is
superpolynomial decay of correlations with respect to Lipschitz observables.

\begin{prop} \label{pro1}
  Suppose the sets $A_{n} = \{\, x \in X : \phi (x) \geq n \,\}$ are
  such that $\phi$ has regular suplevels (Definition~\ref{reg}) and
  that the system $(f,X,\mu )$ satisfies the (SPDC) condition. Then
  for $\mu $-a.e.\ $x$,
  \begin{equation*}
    \lim_{n \to \infty} \frac{\log (\tau (x,A_{n}))}{- \log \mu
      (A_{n})} = 1.
  \end{equation*}
\end{prop}

The proof of Proposition~\ref{pro1} is a direct consequence of the main
result of \cite{G}, which we recall here. Let $g$ be a Borel measurable
function such that $g\geq 0$ on $X$. Consider sublevel sets
\[
V_{r}=\{\,x\in X:g(x)\leq r\,\},
\]
and let us define indicators for the power law
behaviour of the hitting time to the set $V_{r}$ as $r\rightarrow 0$ by
\begin{equation*}
  \overline{H}(x,g)=\limsup_{r\rightarrow 0}\frac{\log \tau
    (x,V_{r})}{-\log (r)}\quad \text{and}\quad
  \underline{H}(x,g)=\liminf_{r\rightarrow 0}\frac{ \log \tau
    (x,V_{r})}{-\log (r)}.  
\end{equation*}
In this way if $\overline{H}(x,g)=\underline{H}(x,g)=H(x,g)$, then $\tau
(x,V_{r})$ scales like $r^{-H(x,g)}$ for small $r$. By analogy with the
definition of local dimension of a measure let us consider
\begin{equation*}
  \overline{d}_{\mu }(g)=\limsup_{r\rightarrow 0}\frac{\log \mu
    (V_{r})}{\log (r)}\qquad \text{and}\qquad \underline{d}_{\mu
  }(g)=\liminf_{r\rightarrow 0} \frac{\log \mu (V_{r})}{\log (r)}.
\end{equation*}
In the following proposition, we deduce that $\overline{H}(x,g)=\underline{H}(x,g)=H(x,g)$ is a typical outcome
in the case where $(f,X,\mu)$ is rapidly mixing.

\begin{prop}[\protect\cite{G}]
  \label{prop.G} Suppose $g\colon X\rightarrow \mathbb{R}^{+}$ is Lipschitz, the system $(f,X,\mu)$ satisfies condition (SPDC), and $\underline{d}_{\mu
  }(g)=\overline{d}_{\mu }(g)=d_{\mu }(g)<\infty $.  Then for
  $\mu$-a.e.\ $x\in X$ it holds that
  \begin{equation*}
    \overline{H}(x,g)=\underline{H}(x,g)=d_{\mu }(g).
  \end{equation*}
\end{prop}

We also use the following elementary fact about real sequences, whose proof
is omitted.

\begin{lemma}
\label{lemmino} Let $r_{n}$ be a decreasing sequence such that $%
r_{n}\rightarrow 0$. Suppose that there is a constant $c>0$ satisfying $%
r_{n+1}>cr_{n}$ eventually as $n$ increases. Let $\tau _{r} \colon \mathbb{R}
\rightarrow \mathbb{R}$ be decreasing. Then
\begin{equation*}
\liminf_{n\rightarrow \infty } \frac{\log \tau _{r_{n}}}{-\log r_{n}}
=\liminf_{r\rightarrow 0}\frac{\log \tau _{r}}{-\log r} \quad \text{and}
\quad \limsup_{n\rightarrow \infty }\frac{\log \tau _{r_{n}}}{-\log r_{n}}
=\limsup_{r\rightarrow 0}\frac{\log \tau _{r}}{-\log r}.
\end{equation*}
\end{lemma}

\begin{proof}[Proof of Proposition~\protect\ref{pro1}]
Consider $V_{r}=\{\,x\in X:\tilde{\phi}(x)\leq r\,\}$, where $\tilde{\phi}$
is related to $\phi $ (and hence the sets $A_{n}$) by (ii) of Definition~\ref{reg}.

By the definition of $V_{r}$ and the assumption on $\tilde{\phi}$, we
have $ V_{\mu (A_{n})^{\beta }}=A_{n}$. Hence,
Proposition~\ref{prop.G} and Lemma~\ref{lemmino} imply that
\begin{equation*}
  d_{\mu }(\tilde{\phi})=\lim_{r\rightarrow 0}\frac{\log \mu
    (V_{r})}{\log (r)} =\lim_{n\rightarrow \infty }\frac{\log \mu
    (V_{\mu (A_{n})^{\beta }})}{\log (\mu (A_{n})^{\beta
    })}=\underset{n\rightarrow \infty }{\lim }\frac{\log \mu
    (A_{n})}{\log (\mu (A_{n})^{\beta })}=\frac{1}{\beta }.
\end{equation*}

By Proposition~\ref{prop.G}, we know that
\begin{equation*}
\lim_{n\rightarrow \infty }\frac{\log (\tau (x,V_{\mu (A_{n})^{\beta }}))}{%
\log \mu (A_{n})^{\beta }}=d_{\mu }(\tilde{\phi})=\frac{1}{\beta}
\end{equation*}%
and hence
\begin{equation*}
\lim_{n\rightarrow \infty }\frac{\log (\tau (x,A_{n}))}{-\log \mu (A_{n})}=1.%
\qedhere
\end{equation*}
\end{proof}

We now complete the proof of Theorem~\ref{thm1}.

\begin{proof}[Proof of Theorem~\protect\ref{thm1}]
First we prove a lower bound to $S_{n}$. Note that the non-integrability
assumption on $\phi $ implies that $\alpha _{\phi }\leq 1$. Since $\phi $ is
non-negative, we have
\begin{equation*}
\lim_{n\rightarrow \infty }\frac{\log (S_{n}(x))}{\log n}\geq
\lim_{n\rightarrow \infty }\frac{\log (\max_{1\leq i\leq n}\phi (f^{i}(x)))}{\log n},
\end{equation*}
and hence from $\frac{\log (\tau (x,A_{n}))}{-\log \mu (A_{n})}\rightarrow 1$, we have
$\frac{\log (\tau (x,A_{n}))}{-\log \mu (A_{n})}\frac{-\log \mu (A_{n})}{
\log n}\rightarrow \alpha _{\phi }$. So $\forall \varepsilon \geq 0$, we
have eventually $n^{\alpha_{\phi} -\varepsilon }\leq \tau (x,A_{n})\leq n^{\alpha_{\phi}
+\varepsilon }$. Furthermore eventually with respect to $n$
\begin{align*}
\max_{1\leq i\leq n}\phi (f^{i}(x))& \geq \max (\{\,i:\tau (x,A_{i})\leq
n\,\}) \\
& \geq \max (\{\,i:i^{\alpha_{\phi} +\varepsilon }\leq n\,\}) \\
& \geq n^{\frac{1}{\alpha_{\phi} +\varepsilon }}-1.
\end{align*}

To get an upper bound on $S_{n}$, let us suppose $\alpha_{\phi} \neq 0$ and
$0<\varepsilon \leq \alpha_{\phi} $. Consider $a(x)=x^{\alpha_{\phi} -\varepsilon }$. Then
by the definition of $\alpha_{\phi} $, we have $\int a(\phi )\,d\mu <\infty $.
Proposition~\ref{aa} implies that $a(S_{n}(x))/n\rightarrow 0$ for almost
every $x$, and hence that
\begin{equation*}
\limsup_{n\rightarrow \infty }\frac{\log S_{n}(x)}{\log n}\leq
\frac{1}{\alpha_{\phi} -\varepsilon }.
\end{equation*}
Since $\varepsilon $ can be taken arbitrarily small, this finishes the proof.
\end{proof}

\section{Proofs of the statements of Sections~\protect\ref{extremes} and \ref{sec.max-hit}}
In this section, we give the proof of results in Section~\ref{extremes}, namely that of
Proposition~\protect\ref{start} on the logarithm law of the hitting time for infinite systems.
We also prove results stated in Section~\ref{sec.max-hit}, namely those that link the hitting time
function with the maxima function.

\subsection{Proofs of results in Sections~\protect\ref{s1}}

\label{sec:proofstart}

\begin{proof}[Proof of Proposition~\protect\ref{start}]
By Proposition~\ref{pro1}, for the induced system it holds that for $\mu_Y$-a.e.\ $x$
\begin{equation}
\lim_{n\rightarrow \infty }\frac{\log \tau (f_Y,x,B_{n})}{\log n}
=\alpha .  \label{sss}
\end{equation}

For the original map $f$ it holds
\begin{equation*}
\tau (f,x,B_{n})=\sum_{i=0}^{\tau (f_Y,x,B_{n})}R((f_Y)^{i}(x)).
\end{equation*}
Hence $\tau (f,x,B_n)$ is a Birkhoff sum of the observable $R$ on the system $(f_Y,Y,\mu_Y),$
applying Theorem~\ref{thm1} we get
\begin{equation*}
\lim_{n\rightarrow \infty }\frac{\log (\tau (f,x,B_{n}))}{\log [\tau
(f_Y,x,B_{n})]}=\frac{1}{\alpha _{R}}
\end{equation*}
from which applying \eqref{sss} we get the statement.
\end{proof}

\subsection{Proofs of results in Section~\ref{sec.max-hit}}\label{sec.proofmaxlaw}

\begin{proof}[Proof of Proposition~\ref{prop.maxhit}.]
First, we suppose there exist $\ell_1(n),\ell_2(n)$ as described
in the proposition for which
\[
\ell_1(n)\leq \tilde{M}_n\leq \ell_2(n),
\]
(eventually, for all large $n$). Since
$\tilde{M}_{n}(x)\leq \ell_2(n)$ implies
$\tau_{\ell_2(n)}(x)\geq n$, it follows that  $\tau_n(x)\geq\ell^{-1}_2(n).$ Now
fix $n\geq N$, and take $u\in[n,n+1]$. It follows that
$\tau_u(x)\geq\ell^{-1}_2(n)\geq\ell^{-1}_2(u-1),$ as $u\to\infty$. A
similar estimate is achieved for the upper bound leading to
$\tau_u(x)\leq \ell^{-1}_1(u+1).$ This proves the first item.

For the second item of the proposition, set
$u=\hat{\ell}^{-1}_1(n),\hat{\ell}^{-1}_2(n)$ accordingly. If
$n\to\infty$ then $u\to\infty$. Hence by using the basic observation
between maxima and hitting times it follows that
$\hat{\ell}^{-1}_2(n)\leq \tilde{M}_n\leq\hat{\ell}^{-1}_1(n)$ for all
$n$ sufficiently large. This completes the proof.
\end{proof}

\begin{proof}[Proof of Proposition~\ref{poi}]
We shall apply Proposition~\ref{prop.maxhit}.  As before, note that
$M_{n}(x)\leq u$ if and only if $\tau^{\phi}_{u}(x)\geq n$. Suppose
that $\limsup_{n\rightarrow \infty }\frac{\log M_{n}}{\log n}=a_1$,
and let $\varepsilon >0$. Then there exists an integer $N_1$ such that
for all $n\geq N_1$ we have $\frac{\log M_{n}}{\log n}\leq
a_1+\varepsilon$.  It follows that for all such $n$, $M_{n}\leq
\ell(n):=n^{a_1+\varepsilon}$. Hence, applying
Proposition~\ref{prop.maxhit}, we have for all sufficiently large $u$:
$\tau_u\geq\ell^{-1}(u-1)=(u-1)^{\frac{1}{a_1+\varepsilon}}.$ By
applying a similar estimate to get the lower bound we achieve (for all
sufficiently large $u$) that
\[
(u-1)^{\frac{1}{a_1+\varepsilon}}\leq\tau^{\phi}_u\leq
(u+1)^{\frac{1}{a_2+\varepsilon}}.
\]
Hence, taking logarithms we get
\[
\liminf_{u\to\infty}\frac{\log \tau_u}{\log u}\geq\frac{1}{a_1},\;\textrm{and}\;
\limsup_{u\to\infty}\frac{\log \tau_u}{\log u}\leq\frac{1}{a_2}.
\]

To get equality for the $\liminf$, we know that for all $\varepsilon>0$, we have $M_{n}\geq
n^{a_1-\varepsilon}$ infinitely often.  Hence $\tau _{u_n}\leq
(u_n)^{\frac{1}{a_1-\varepsilon }}$ infinitely often along the sequence
$u_n=n^{a_1-\varepsilon}$. Thus by taking logarithms we obtain $\lim
\inf_{u\rightarrow \infty }\frac{\log \tau^{\phi}_{u}}{\log u}\leq
\frac{1}{a_1}$.  This establishes the implication
\begin{equation*}
\limsup_{n\rightarrow \infty } \frac{\log [M_{n}(x)]}{\log n}=a_1\;
\implies\;   \liminf_{u\rightarrow \infty } \frac{\log
[\tau^{\phi}_{u}(x)]}{\log u}=\frac{1}{a_1}.
\end{equation*}
By a symmetric argument, we also establish
that for given $a_2>0$, and $x\in X$,
\begin{equation*}
  \liminf_{n\rightarrow \infty } \frac{\log [M_{n}(x)]}{\log n}=a_2
  \; \implies\; \limsup_{u\rightarrow \infty }
  \frac{\log [\tau^{\phi}_{u}(x)]}{\log u}=\frac{1}{a_2}.
\end{equation*}
The converse implications follow in a similar way following the proof
of Proposition~\ref{prop.maxhit}.
Hence in the case $a_1=a_2$, and when either limit exists,
we obtain the final limit statement in Proposition~\ref{poi}.
\end{proof}

\begin{proof}[Proof of Corollary~\ref{t2}]
  Consider the observable $\phi $ and the hitting time scaling
  behaviour of suplevels $B_{n}=\{\,x\in X:\phi (x)\geq
  n\,\}$. Restricting to countably many radii and considering that
  $\tau^{\phi}_{u}(x)$ is increasing in $u$,
  \begin{equation*}
    \lim_{u\rightarrow \infty }\frac{\log [\tau^{\phi}_{u}(x)]}{\log u}
    =\lim_{n\rightarrow \infty }\frac{\log [\tau (f,x,B_{n})]}{\log
      n}.
  \end{equation*}
  By Proposition~\ref{start} we then get that
  \[
  \lim_{r\rightarrow
    \infty }\frac{\log [\tau^{\phi}_{u}(x)]}{\log u}=\lim_{n\rightarrow
    \infty }\frac{\log [\tau (f,x,B_{n})]}{\log n}=\frac{\alpha _{\phi
  }}{\alpha _{R}}
  \]
  holds for $\mu$-a.e.\ $x\in X$. Applying
  Proposition~\ref{poi} we directly get the statement.
\end{proof}

\begin{proof}[Proof of Lemma~\ref{lem.tent}.]
From Proposition~\ref{prop.maxhit} it suffices to estimate an expression for the asymptotic inverse function of
\[
g(x):=\log x+c\log\log x,\quad c>0.
\]
Consider the function $\tilde{g}(x)=x^{-a}e^{x}$, for some $a>0$. If we compute $(g\circ \tilde{g})(x)$, we obtain
\[
g(\tilde{g}(x))=x-a\log x+c\log x+c\log\left(1-\frac{a\log x}{x}\right).
\]
If $a>c$, then $g(\tilde{g}(x))<x$, as $x\to\infty$, and similarly if $a<c$, then $g(\tilde{g}(x))>x$, as $x\to\infty$.
Hence for all $\varepsilon>0$, and all sufficiently large $x$, we have
\[
\frac{e^x}{x^{c+\varepsilon}}\leq g^{-1}(x)\leq \frac{e^x}{x^{c-\varepsilon}}.
\]
Applying Proposition~\ref{prop.maxhit}, and then taking logarithms gives the result.
\end{proof}

\section{Proof of Theorem~\ref{thm:infiniteSBC} and Proposition~\ref{prop:SnR}.} \label{sec:proofinfiniteSBC}
In this section we prove Theorem~\ref{thm:infiniteSBC} on SBC results for infinite systems.
We begin with a proof of Proposition~\ref{prop:SnR}.

\begin{proof}[Proof of Proposition~\ref{prop:SnR}.]

 Recall that $\phi$ is such that $\mu\{\phi(x)=n\}\sim n^{-\beta-1}$.
  For the proof of the upper bound let $a(t)=t^{\beta }(\log
  t)^{-1-\varepsilon }$. Then
  \begin{equation*}
    \int a(\phi )\,\mathrm{d}\mu <\infty ,
  \end{equation*}
  since
  \begin{align*}
    \int a(\phi )\,\mathrm{d}\mu & =\sum_{k=1}^{\infty }a(k)\mu
    \{\,x:\phi (x)=k\,\} \\ & \sim \sum_{k=1}^{\infty }a(k)k^{-\beta
      -1}=\sum_{k=1}^{\infty }\frac{1}{ k(\log k)^{1+\varepsilon
    }}<\infty .
  \end{align*}

  By Proposition~\ref{aa}, we have for almost all $x$ that
  \begin{equation*}
    \frac{a(S_{n}(x))}{n}\rightarrow 0.
  \end{equation*}
  In particular $a(S_{n})<n$ if $n$ is large. Then almost surely,
  \begin{equation*}
    (S_n)^{\beta}(\log S_n)^{-1-\varepsilon }<n.
  \end{equation*}
  By asymptotic inversion, we therefore have for almost all $x$ that
  \begin{equation*}
    S_{n}(x)\leq Cn^{1/\beta}(\log n)^{1/\beta +
      \varepsilon/\beta }
  \end{equation*}
  for all large $n$. As $\varepsilon $ is arbitrary, we may take $C=1$
  and replace $\alpha \varepsilon $ by $\varepsilon $, which proves
  the upper bound.

  For the proof of the lower bound we will use that $(f,X,\mu )$ is a
  Gibbs--Markov map satisfying assumptions~\ref{G-M1}--\ref{G-M6}. We
  will use the following lemma, which we prove in the Appendix.

  \begin{lemma}
    \label{slem.gamma}
    Assume that \ref{G-M1}--\ref{G-M6} hold. Suppose
    $\gamma_n\to\infty$ is a monotone sequence. Let
    \begin{equation*}
      P_n:= \mu \{\, x : \phi (f^j(x)) < \gamma_n \text{ for all } j <
      n \,\}.
    \end{equation*}
    Then there exists $D_0, D_1 > 0$ such that
    \[
      P_n \leq D_1 \left(1 - D_0 \gamma^{-\beta}_n\right)^n.
    \]
  \end{lemma}

  Using Lemma~\ref{slem.gamma}, we now put $\gamma_n=
  n^{\frac{1}{\beta}} (\log n)^{-\frac{1}{\beta}-\varepsilon}$. It
  follows that
  \begin{equation*}
    P_n\leq D_1 (1 - D_0 n^{-1} (\log n)^{1+\varepsilon\beta} )^n,
  \end{equation*}
  For a sequence $a_n$ such that $na_n\to 0$, an elementary estimate
  gives
  \begin{equation*}
    (1-a_n)^n=\exp\{n\log(1-a_n)\}\leq\exp\{-a_nn\}.
  \end{equation*}
  In the case $a_n=D_0\gamma^{\beta}_n$ we obtain
  \begin{equation*}
    P_n < D_1 \exp(-D_0(\log n)^{1+\frac{\varepsilon}{\alpha}})= O(n^{-2}),
  \end{equation*}
  for large $n$. Since $P_n$ is summable, it follows by the First
  Borel--Cantelli Lemma that if $n$ is large, there is always a $j <
  n$ with
  \begin{equation*}
    \phi (f^j(x)) > n^{\frac{1}{\beta}} (\log n)^{-\frac{1}{\beta} -
      \varepsilon}.
  \end{equation*}
  This proves the lower bound for $\max \{\, \phi (f^j(x)) : 0 \leq j
  < n \,\}$, concluding the proof of Proposition~\ref{prop:SnR}.
\end{proof}

We now prove Theorem~\ref{thm:infiniteSBC}.

\begin{proof}[Proof of Theorem~\ref{thm:infiniteSBC}.]
  The proof consists of the following steps. First we obtain an almost
  sure asymptotic between the inducing time $n$, and the clock time
  $S^{R}_n(x)=\sum_{i=1}^{n}R_i(x)$, in the limit $n\to\infty$.  For
  systems preserving an infinite measure, the asymptotics of the
  return function $R(x)$ are important in this step.  In the second
  step, we then use the strong Borel--Cantelli property of the induced
  system to get quantitative bounds on the recurrence statistics,
  namely equation \eqref{eq.quant}.

  Let
  \[
  Q(n,x) = S_n (R)(x) = \sum_{i=0}^{n-1} R(f^{i}_{Y}(x)).
  \]
  Then $f^{Q(n,x)}(x)=f^{n}_{Y}(x)$.  We begin by using
  Proposition~\ref{prop:SnR}, which tells us that
  \begin{equation} \label{eq:Qupperbound}
  Q (n,x) \leq n^{\frac{1}{\beta}} (\log
  n)^{\frac{1}{\beta} + \varepsilon},
  \end{equation}
  and
  \[
  n^{\frac{1}{\beta}}(\log n)^{-\frac{1}{\beta}-\varepsilon} \leq \max
  \{\, R (f^j(x)) : 0 \leq j < n \,\} < Q (n,x)
  \]
  for all large $n$.

  We consider the first item of the theorem. Let
  $q_n=p_{[n^{\alpha+\varepsilon}]}$ with $\varepsilon \in (0,
  \varepsilon_1]$. Since $\varepsilon \leq \varepsilon_1$ we have
  $\sum_n \mu (q_n) = \infty$, hence $\sum_{k} \mu(p_{n_k}) = \infty$,
  and so $\{p_{n_k}\}$ forms a strong Borel--Cantelli sequence with
  respect to $f_Y$. We obtain
  \[
  \lim_{N\to\infty}
  \frac{\sum_{k=1}^{N}q_n(f_{Y}^k(x))}{\sum_{k=1}^{N}\mu(q_n)} = 1,
  \qquad \text{for $\mu_Y$-almost every } x \in Y.
  \]
  Consider $\ell$ such that $Q(n,x)\leq\ell< Q(n+1,x)$. By
  \eqref{eq:Qupperbound}, there are infinitely many $n$ such that
  $\ell$ satisfies
  \begin{equation} \label{eq:lbounds}
    n^{\frac{1}{\beta}}(\log n)^{-\frac{1}{\beta} - \varepsilon} <
    \ell < (n+1)^{\frac{1}{\beta}}(\log (n + 1))^{\frac{1}{\beta} +
      \varepsilon}.
  \end{equation}
  We remark that in some cases, when $Y$ is a Darling--Kac set, then
  we get an improved lower bound on $\ell$
  \cite[Theorem~4]{AaronsonDenker}, which in turn leads to an improved
  upper bound in \eqref{eq.quant} of Theorem~\ref{thm:infiniteSBC}.

  By monotonicity of $p_n$, and noting that $p_k(f^k(x))=0$ when
  $k\neq Q(n,x)$, we have
  \begin{equation} \label{eq:ergodicsums}
    \sum_{k=1}^{n}q_k(f_{Y}^k(x)) \leq \sum_{k=1}^{\ell} p_k (f^k(x))
    + \sum_{k=1}^{N_0}p_1(f_Y^k (x)),
  \end{equation}
	for some $N_0=N_0(x)$.
  By a rearrangement and division by $\sum_{k=1}^{n}\mu(q_k)$, we
  obtain
  \[
    \frac{\sum_{k=1}^{\ell}p_k(f^k(x))} {\sum_{k=1}^{n}\mu(q_k)} \geq
    \frac{ \sum_{k=1}^{n} q_k(f_{Y}^k(x))-\sum_{k=1}^{N_0} p_1 (f_Y^k
      (x)) } { \sum_{k=1}^{n}\mu(q_k) }.
  \]
  As $n\to\infty$ (and $q\to\infty$) the right-hand bracket is
  $1+o(1)$ due the strong Borel--Cantelli property of the sequence
  $q_k$ with respect to $f_Y$.  Using the bounds on $\ell$ in
  \eqref{eq:lbounds}, and the monotonicity of the sequence $\{p_n\}$
  we obtain, for infinitely many $n$ that
  \[
    \sum_{k=1}^{\ell}p_k(f^k(x))=(1+o(1))\sum_{k=1}^{n}\mu(q_k)\geq
    \sum_{k=1}^{\ell^{\frac{1}{\alpha+\varepsilon}}}
    \mu(p_{k^{\alpha+\varepsilon}}).
  \]
  This leads to the conclusion that
  \begin{equation*}
    \liminf_{n\to\infty}\frac{\sum_{k=1}^{n} p_k (f^k
      (x))}{\sum_{k=1}^{n^{\frac{1}{\alpha+\varepsilon}}}
      \mu(p_{k^{\alpha+\varepsilon}})}\geq 1
  \end{equation*}
  holds for $\mu_Y$-almost every $x \in Y$. Clearly, this estimate
  then also holds for $\mu$-almost every $x \in X$, since $\mu$-almost
  every $x$ has $f^k (x) \in Y$ for some $k$ and for all $A\subset X$
  \[
  \mu (A) = \sum_{n=0}^\infty \mu_Y ((f^{-n}A) \cap \{R > n\})).
  \]

  To get an upper bound, similar to \eqref{eq:ergodicsums}, we write
  \[
  \sum_{k=1}^\ell p_k (f^k (x)) = \sum_{j=1}^n p_{Q(j,x)} (f^{Q(j,x)}
  (x)) = \sum_{j=1}^n p_{Q(j,x)} (f_Y^j (x)).
  \]
  Hence
  \[
  \frac{\sum_{k=1}^{\ell} p_k (f_k (x))}{(\sum_{k=1}^{n} \mu (q_k))} =
  \frac{\sum_{j=1}^n p_{Q(j,x)} (f_Y^j (x))}{\sum_{j=1}^n \mu
    (p_{Q(j,x)})}.
  \]
  By the strong Borel--Cantelli property for $(f_Y,Y,\mu_Y)$, we have
  \[
  \frac{\sum_{j=1}^n p_{Q(j,x)} (f_Y^j (x))}{\sum_{j=1}^n \mu
    (p_{Q(j,x)})} \to 1
  \]
  for almost every $x$.  Using again the bounds on $\ell$ in
  \eqref{eq:lbounds}, and the monotonicity of the sequence $\{p_n\}$,
  we obtain, for large enough $n$ that
  \begin{equation}\label{eq.bcestqk}
    \sum_{k=1}^{n}\mu(q_k) \leq
    \sum_{k=1}^{\ell^{\frac{1}{\alpha-\varepsilon}}}
    \mu(p_{k^{\alpha-\varepsilon}}).
  \end{equation}
  This leads to the estimate that
  \begin{equation*}
    \limsup_{n\to\infty}\frac{\sum_{k=1}^{n} p_k (f^k
      (x))}{\sum_{k=1}^{n^{\frac{1}{\alpha-\varepsilon}}}
      \mu(p_{k^{\alpha-\varepsilon}})}\leq 1
  \end{equation*}
  holds for $\mu_Y$-almost every $x \in Y$, and therefore also for
  $\mu$-almost every $x \in X$. This proves the first item.
	
  To prove the second item of Theorem~\ref{thm:infiniteSBC} we repeat
  the estimates above. This time we use the First Borel--Cantelli
  Lemma to deduce first of all that if $\sum_k \mu (q_k) < \infty$
  then $\sum_{k=1}^{n} q_k (f^{k}_Y (x)) < \infty$.  Using equation
  \eqref{eq.bcestqk}, and for all $\varepsilon>0$ we obtain the
  eventual bound (in $n$),
  \begin{equation*}
    \sum_{k=1}^{n} p_k (f^k(x))\leq
    \sum_{k=1}^{n^{\frac{1}{\alpha-\varepsilon}}}
    \mu(p_{k^{\alpha-\varepsilon}}).
  \end{equation*}	
  However by the assumption of item (2), the right hand sum is
  uniformly bounded, and hence the First Borel--Cantelli Lemma implies
  that for $\mu$-a.e.\ $x\in X$
  \[
  \sum_{k=1}^{n} p_k (f^k(x))<\infty.
  \]
\end{proof}

\section{Proof of limit laws for maxima and hitting times}\label{sec.bcmaxhit-proof}

Regarding the almost sure growth of $M_n$ for infinite systems, in
this section we prove Theorem~\ref{thm:infinitemaxgibbs}. The main
idea is to use directly Proposition~\ref{prop:SnR}, and the structure
of the induced system $(f_Y,Y,\mu_Y)$.

\begin{proof}[Proof of Theorem~\ref{thm:infinitemaxgibbs}.]
  We use Proposition~\ref{prop:SnR}, and note that $f_Y$ and the
  return time function $R \colon Y\to\mathbb{N}$ satisfy the
  assumptions~\ref{G-M1}--\ref{G-M6} (i.e.\ with $\phi$ in place of
  $R$ in \ref{G-M6}).  We get that for all $\varepsilon > 0$, and
  $\mu_Y$-almost all $x\in Y$ there is an $n_0$ such that
  \[
  n^{\frac{1}{\beta}}(\log n)^{-\frac{1}{\beta}-\varepsilon} \leq \max
  \{\, R_j (x) : 0 \leq j < n \,\} < \sum_{j=0}^n R_j (x) \leq
  n^{\frac{1}{\beta}} (\log n)^{\frac{1}{\beta} + \varepsilon}
  \]
  holds for all $n > n_0$. Now, let $\psi \colon
  (0,\infty)\to[0,\infty)$ be a decreasing function, and put
  \[
  M_n (x) = \max \{\, \psi (d (f^j x, \tilde{x})) : 0 \leq j < n \,\}.
  \]
  Suppose now that $n$ is fixed. In the case $\tilde{x}\in Y$, we have
  $M_n(x) = \hat{M}_k(x)$, where
  \[
  \hat{M}_k(x):=\max_{j\leq k(x)}\psi(d(f^{j}_{Y}(x),\tilde{x})),
  \]
  and $k(x)$ is the largest such $k$ for which $n \geq
  \sum_{j=0}^{k-1} R_j (x)$.  By above, we have for $\mu_Y$-almost all
  $x$ that
  \[
  k^\alpha (\log k)^{-\alpha - \varepsilon} \leq \max\{ R_0 , \ldots,
  R_{k-1} \} < n \leq k^\alpha (\log k)^{\alpha + \varepsilon}
  \]
  when both $n$ and $k = k(x)$ are large. This implies that
  \[
  n^{1/\alpha} (\log k)^{-1 - \frac{\varepsilon}{\alpha}} \leq k \leq
  n^{1/\alpha} (\log k)^{1 + \frac{\varepsilon}{\alpha}},
  \]
  and using that $k \leq n$, we obtain
  \begin{equation} \label{eq:n-k-relation}
  n^{1/\alpha} (\log n)^{-1 - \frac{\varepsilon}{\alpha}} \leq k \leq
  n^{1/\alpha} (\log n)^{1 + \frac{\varepsilon}{\alpha}}.
  \end{equation}

  Since the system $(f_Y,Y,\mu_Y)$ has exponential decay of
  correlations, we can use \cite[Proposition~3.4]{HNT} to get refined
  bounds on the almost sure growth of the maximum function
  $\hat{M}_k(x)$. That is,
  \[
  \psi \biggl( \frac{(\log k)^3}{k} \biggr) \leq \hat{M}_k(x) \leq
  \psi \biggl( \frac{1}{k (\log k)^{1+\varepsilon}} \biggr),
  \]
  for $\mu_Y$-a.e.\ $x\in Y$, and for all $\varepsilon>0$. Combining this with
  \eqref{eq:n-k-relation} and using that $M_n(x) = \hat{M}_k(x)$, we
  obtain
  \[
  \psi \biggl(\frac{(\log n)^{4 + \varepsilon}}{n^{1/\alpha}} \biggr)
  \leq M_n(x) \leq \psi \biggl(\frac{1}{n^{1/\alpha} (\log n)^{2 +
      \varepsilon}} \biggr). \qedhere
  \]
  These bounds pass on to $\mu$-a.e.\ $x\in X$, since $\mu$-a.e.\ $x$
  has $f^k(x)\in Y$ for some $k$.
\end{proof}

We are now readily to prove Theorem~\ref{thm:maximaint} for the
explicit family of intermittent maps $(f_{\alpha},X,\mu)$.

\begin{proof}[Proof of Theorem~\ref{thm:maximaint}]
To prove this result, we consider three cases: $\tilde{x}\in Y$,
$\tilde{x}\in(0,1/2)$, and also $\tilde{x}=0$. In the case
$\tilde{x}\in Y$, we notice that contributions to successive maxima
only occur once orbits return to $Y$, and hence the
statistics of the induced system $(f_Y,Y,\mu_Y)$ apply to obtain growth
rates for $M_n$. In the case $\tilde{x}\not\in Y\cup\{0\}$, we show
that the inducing set $Y$ can be enlarged to a new set
$\tilde{Y}\supset Y$, with $\tilde{x}\in\tilde{Y}$, and that the
corresponding induced system satisfies \ref{G-M1}--\ref{G-M6}. In the case
$\tilde{x}=0$, we use explicit tracking of the orbits outside of $Y$
to deduce the growth rate of $M_n$.

Consider first the case  $\tilde{x}\in Y$. Here, we just apply
Theorem~\ref{thm:infinitemaxgibbs} directly to
this system, and obtain immediately
  \[
  \psi \biggl(\frac{(\log n)^{4 + \varepsilon}}{n^{1/\alpha}} \biggr)
  \leq M_n(x) \leq \psi \biggl(\frac{1}{n^{1/\alpha} (\log n)^{2 +
      \varepsilon}} \biggr),
  \]
for $\mu$-a.e.\ $x\in [0,1]$, and $\tilde{x} \in Y$.

 So suppose now that $ \tilde{x}\in(0,1/2)$. We now enlarge the inducing set
  $Y$ to the set
  \[
  \tilde{Y}=Y\cup\left(\bigcup_{j=1}^{m}[x_j,x_{j-1}]\right),
  \]
  with $m$ the smallest integer so that $\tilde{x}$ lies in the
  interior of $\tilde{Y}$. Recalling that $x_n\sim (\alpha n)^{-1/\alpha}$,
let $W_{j}=[x_{j},x_{j-1}]$, and define a
  new (first) return time function $\tilde{R}$ via
  $\tilde{R}|_{Y_i}=1$ with $i\leq m$, $\tilde{R}|_{Y_i}=i-m$ for
  $i>m$, and $\tilde{R}\mid_{W_i}=1$.  The corresponding induced map
  $\tilde{f}(x)=f^{\tilde{R}}(x)$ satisfies \ref{G-M1}--\ref{G-M6}.  For any $x$
  and $j$, we have
  \[
  \tilde{R}_j (x) \leq R_j (x) \leq \tilde{R}_j (x) + m,
  \]
  where, as before, $R_j (x)$ denotes a return time with respect to
  $[1/2, 1]$. Hence, Lemma~\ref{prop:SnR} holds for $\tilde{R}$ as
  well (this also follows by Hopf's ergodic theorem), and this lets us
  prove the result in the same way as for the case $\tilde{x} \in
  [1/2,1]$.

In the case $\tilde{x}=0$, we have
  \[
  M_n (x)= \max \{\, \psi (x_{R_j (x)}) : 0 \leq j < k \,\},
  \]
 where $k=k(x)$ is the largest such $k$ for which $n \geq \sum_{j=0}^{k-1} R_j (x)$.
We can now follow step by step the proof of Theorem~\ref{thm:infinitemaxgibbs},
namely equation \eqref{eq:n-k-relation}
to deduce the relevant bounds on $M_n$ as stated in Theorem~\ref{thm:maximaint}.
\end{proof}

\section{Dynamical run length problems---proof of Theorem~\ref{thm_runlength}} \label{sec:runlength}

In this section, we prove Theorem~\ref{thm_runlength} for the
family of intermittent maps $(f_{\alpha},X,\mu)$. There is a natural
link between hitting times and the run length function as we now make
concrete.  Namely, consider a target point $\tilde{x}$ with a target
ball $B_{\varepsilon}(\tilde{x})$, and recall that the hitting time of
a point $x \in S^{1}$ is defined by
\begin{equation*} 
  \tau_{\varepsilon}(x,\tilde{x}) = \min \{\, n\geq 1: f_\alpha^n (x) \in
  B_{\varepsilon}(\tilde{x}) \,\}.
\end{equation*}
Meanwhile, let $T(x)=2x\mod 1$ on $S^{1}$, and for every $x\in S^{1}$,
denote $x=\sum_{i=1}^{\infty}\frac{x_{i}}{2^{i}}$ with $x_{i}=0$
(resp.\ 1) if and only of $T^{i-1}(x)\in[0,1/2)$
  (resp.\ $[1/2,1)$). Then we define the \emph{binary symbolic coding
      distance}
\begin{equation*} 
  \tilde{d}(x,y)=2^{-n^{*}(x,y)},~~\forall x,y\in S^{1},
\end{equation*}
where $n^{*}(x,y):=\min\{i\in\mathbb{N},~~x_{i}\neq y_{i}\}$.

With these conventions, we commence with the following lemma.
\begin{lemma}\label{Lem_hittingandrunlength}
  For every $x\in S^{1}$ and every $n\in\mathbb{N}$, we have
    \begin{enumerate}
    \item[(i)] $\displaystyle 2^{-\xi^{(1)}_{n}(x)} = \min_{1\leq i\leq
      n} \max \{ \tilde{d}(f_\alpha^i (x),1), 2^{-(n-i)} \}$;
    \item[(ii)] $\displaystyle \min_{1\leq i\leq \tau_{2^{-n}} (x,1)}
      \tilde{d}(f_\alpha^{i}(x), 1) \leq 2^{-n}$;
    \item[(iii)] $\displaystyle \min_{1\leq i\leq \tau_{2^{-n}} (x,1) -
      1} \tilde{d}(f_\alpha^{i}(x), 1) \geq 2^{-n}$,
    \end{enumerate}
    where $\tilde{d}(\cdot,\cdot)$ is the (binary) symbolic coding distance.
\end{lemma}

\begin{proof}
  The lemma follows directly from the definitions of
  $\tau_{2^{-n}}(x,1)$, and binary symbolic coding distance $\tilde{d}(\cdot,\cdot)$.
\end{proof}

Recall that $Y=[1/2,1)$ and for each $\tilde{x}\in Y$, we analogously define
the hitting time on the induced map $f_{Y}$ by
\begin{equation*} 
  \hat{\tau}_{\varepsilon}(x,\tilde{x}) := \min \{\, n \geq 1 : f_{Y}^n
  (x) \in B_{\varepsilon} (\tilde{x}) \,\}.
\end{equation*}
There is a relationship between $\tau$ and $\hat{\tau}$, that is
\begin{equation*} 
  \tau_{2^{-n}} (y,1) = \sum_{j=0}^{\hat{\tau}_{2^{-n}}(y,1)} R
  \left( f_{Y}^{j} (y)\right), \qquad \text{for all } y\in Y.
\end{equation*}

%

We are now ready to prove Theorem~\ref{thm_runlength}.

\begin{proof}[Proof of Theorem~\ref{thm_runlength}]
  We will first prove the first assertion \eqref{equ_runlength1} of
  Theorem~\ref{thm_runlength}. For any $\varepsilon > 0$ and Lebesgue
  almost every $y\in Y$, by Proposition~\ref{prop:SnR}, we have
  \[
  (\hat{\tau}_{2^{-n}}(y,1))^{\alpha-\varepsilon} \leq
  \tau_{2^{-n}}(y,1) = \sum_{j=0}^{\hat{\tau}_{2^{-n}}(y,1)} R
  \circ f_{Y}^{j}(y) \leq
    (\hat{\tau}_{2^{-n}}(y,1))^{\alpha+\varepsilon}
  \]
  if $n$ is large enough. Together with
  Corollary~\ref{cor.logithmnlaw} and Remark~\ref{rmk.logithmnlaw}, we conclude that
  \[
  \lim_{n\to\infty} \frac{ \log \tau_{2^{-n}}(y,1)}{-\alpha \log
    2^{-n}} = 1, \qquad \text{for Lebesgue almost every } y\in[1/2,1).
  \]
  Note also that for Lebesgue almost every $x\in S^{1}$, there is
  always a $y\in[1/2,1)$ such that
    $\tau_{2^{-n}}(x,1)=\tau_{2^{-n}}(y,1)$. This implies that
  \[
  \lim_{n\to\infty} \frac{ \log \tau_{2^{-n}}(x,1)}{-\alpha \log
    2^{-n}} = 1, \qquad \text{for Lebesgue almost every } x \in S^{1}.
  \]
  Hence, we have
  \[
  n \sim \frac{1}{\alpha} \log_{2}^{\tau_{2^{-n}}
    (x,1)},\qquad \text{as } n \to \infty.
  \]
  Together with assertions (ii) and (iii) in
  Lemma~\ref{Lem_hittingandrunlength},
  this implies that
  \[
  n^{-1/\alpha - \varepsilon} \leq \min_{1\leq i\leq n}\tilde{d}(f_\alpha^{i}(x),1)\leq
  n^{-1/\alpha + \varepsilon},
  \]
  if $n$ is large enough.

  Finally, for any $\varepsilon>0$, and sufficiently large $n$, we
  have
  \begin{align*}
    n^{-1/\alpha - \varepsilon} \leq \min_{1\leq i\leq n}
    \tilde{d}(f_\alpha^i(x),1)&\leq \min_{1\leq i \leq n} \max \{ \tilde{d}(f_\alpha^i
    (x), 1),2^{-(n-i)}\}\\ &\leq \min_{1\leq i \leq n^{1-\varepsilon}}
    \max \{ \tilde{d}(f_\alpha^i (x, 1)),2^{-(n-i)}\}\\ &\leq \min_{1\leq i
      \leq n^{1-\varepsilon}} \max \{ \tilde{d}(f_\alpha^i (x,
    1)),2^{-n+n^{1-\varepsilon}}\}\\ &\leq \min_{1\leq i \leq
      n^{1-\varepsilon}} \max \{ \tilde{d}(f_\alpha^i (x,1)),2^{-n/2} \} \\ &=
    \max\{\min_{1\leq i \leq n^{1 - \varepsilon}} \{
    \tilde{d}(f_\alpha^{i}(x),1),2^{-n/2}\}\}\\ &\leq
    \max\{(n^{(1-\varepsilon)})^{-1/\alpha + \varepsilon}, 2^{-n/2}\}
    \leq n^{-(1-\varepsilon)/\alpha + \varepsilon}.
  \end{align*}
  Since $\varepsilon$ is arbitrary, this implies that
  \[
  \lim_{n \to \infty} \frac{\log_{2} \left(\min_{1\leq i\leq n}\left(\max \{ \tilde{d}
    (f_\alpha^{i} (x),1),2^{-(n-i)}\}\right)\right) }{\log_{2} n} = -\frac{1}{\alpha}.
  \]
  Together with Item (i) of Lemma~\ref{Lem_hittingandrunlength},
  we have
  \[
  \lim_{n\to\infty}\frac{\xi^{(1)}_{n}(x)}{\log_{2}n}=\frac{1}{\alpha},
  \qquad \text{for Lebesgue almost every } x \in S^{1},
  \]
  which is \eqref{equ_runlength1}.

  We will now prove the second assertion \eqref{equ_runlength0} in Theorem~\ref{thm_runlength}. As before, we put
  $x_{0}=1/2$, and $x_{n+1} = f_\alpha^{-1} (x_{n}) \cap [0,1/2)$. For
  $x \in [0,1/2)$ we have $R(x) = n$ if and only if
  $x\in[x_{n},x_{n-1})$. Moreover, we have $x_{n} \sim
  n^{-\frac{1}{\alpha}}$, so $R (x) \sim |x|^{-\alpha}$. It is
  now clear that
  \begin{equation} \label{equ_extreme}
    \xi^{(0)}_{n}(x)=\max_{0\leq i\leq n-1} \{ \min \{ R (f_\alpha^{i}(x)),
    n-i \} \}.
  \end{equation}

  Let
  \[
  M_{n}:=\max_{0\leq i\leq n-1}\{R(f_\alpha^{i}(x))\}.
  \]
  By Item $(1)$ of Theorem~\ref{thm:maximaint}, for any $\varepsilon>0$, and almost
  every $x\in S^{1}$, we have
  \begin{equation*} 
    n^{1-\varepsilon}\leq M_{n}(x)\leq n^{1+\varepsilon}.
  \end{equation*}
  if $n$ is large enough.
  Therefore, we have
  \begin{align*}
    n^{1 + \varepsilon}\geq M_{n}(x)&\geq \max_{0\leq i\leq n-1} \{ \min
    \{ R (f_\alpha^{i}(x)), n - i \} \}\\ &\geq \max_{0\leq i \leq
      \frac{n - 1}{1 + \varepsilon}} \{ \min \{ R (f_\alpha^{i}(x)),n-i
    \} \} \\ & \geq \max_{0 \leq i\leq \frac{n - 1}{1 + \varepsilon}}
    \Bigl\{ \min \Bigl\{ R (f_\alpha^{i}(x)),\frac{\varepsilon n}{1 +
      \varepsilon} \Bigr\} \Bigr\}\\ &\geq \min \Bigl\{ \max_{0 \leq i
      \leq \frac{n - 1}{1 + \varepsilon}} \{ R (f_\alpha^{i} (x))
    \},\frac{\varepsilon n}{1 + \varepsilon} \Bigr\}\\ &\geq \min \Bigl\{
    \Bigl(\frac{n}{1 + \varepsilon} \Bigr)^{1 - \varepsilon},\frac{\varepsilon
      n}{1 + \varepsilon} \Bigr\} \\ &= \Bigl( \frac{n}{1 + \varepsilon}
    \Bigr)^{1 - \varepsilon},
  \end{align*}
  when $n$ is large enough.  Together with \eqref{equ_extreme}, this
  implies that for Lebesgue almost every $x\in S^{1}$,
  \[
  \lim_{n\to\infty} \frac{\log \xi^{(0)}_{n}(x)}{\log n}=1,
  \]
  as was to be proved.
\end{proof}

\section{Birkhoff sums of infinite observables, systems with strongly
oscillating behaviour} \label{oscill}

In this section we exhibit dynamical systems whose Birkhoff sums $S_{n}$ of an
infinite observable have a strongly oscillating bevaviour in the
sense that
\[
\liminf_{n \to \infty} \frac{\log S_{n}(x)}{\log n} < \limsup_{n \to
  \infty} \frac{\log S_{n}(x)}{\log n}.
\]
The behaviour oscillates between two different power laws. We will
also see examples with the same behaviour having polynomial decay of
correlations (SPDC). In particular the examples we use allow us to prove
Theorem~\ref{the:thereareoscillations} in Section~\ref{oscillsec}.
In Theorem~\ref{thm1}, we showed that this is optimal in
some sense: systems with decay of correlations faster than any
polynomial and sufficiently regular observables cannot have such a
strongly oscillating behaviour for observables diverging to $\infty$
at a power law speed.

\subsection{Birkhoff sums of infinite observables and rotations}

Let us recall the definition of Diophantine type of an irrational
number.

\begin{defn} \label{type}
  Given an irrational number $\theta $ we define the Diophantine
  \emph{type} of $\theta $ as the following (possibly infinite)
  number.
  \begin{equation*}
    \gamma (\theta ) = \inf \{\, \beta :\liminf_{q\rightarrow \infty
    }q^{\beta }\Vert q\theta \Vert >0 \,\}.
  \end{equation*}
\end{defn}

Every real number has Diophantine type $\geq 1$. The set of numbers of
type $1$ is of full measure; the set of numbers of type $\gamma $ has
Hausdorff dimension $\frac{2}{\gamma +1}$. There exist numbers of
infinite type, called \emph{Liouville numbers}; the set of which is
dense, uncountable and has zero Hausdorff dimension.
\begin{prop} \label{ooo}
  Let the system $(f_{\theta }, S^1,\mathrm{Leb})$ be the rotation of the circle
  $S^{1}$ by the angle $\theta $ of Diophantine type $\gamma $.
  Consider $\beta \geq 1$ and the non-integrable observable $\psi
  (x)=[d (x,0)]^{-\beta }$. When $\beta >1,$ and $\frac{1}{\gamma
  }<1-\frac{1}{\beta }$ we have
  \begin{equation*}
    \liminf_{n\rightarrow \infty }\frac{\log S_{n}^{\psi }(x)}{\log
      n}\leq 1+ \frac{\beta }{\gamma }<\beta \leq
    \limsup_{n\rightarrow \infty }\frac{\log S_{n}^{\psi }(x)}{\log
      n},
  \end{equation*}
  for Lebesgue almost every $x \in S^1$.
\end{prop}

Before the proof we recall some hitting time results on circle
rotations and relation with minimal distance iterations: the behaviour
of hitting time in small targets for circle rotations with angle
$\theta $ depends on the Diophantine type of irrational $\theta$. We
state the following, where we recall that $\overline{H},\underline{H}$
are defined in Section~\ref{sec:introduction}.
\begin{lemma}[\cite{kimseo}] \label{KS...}
  If $f_{\theta }$ is a rotation of the circle, $y$ a point on the
  circle and $\gamma $ is the Diophantine type of $\theta$ then for
  Lebesgue almost every $x$
  \begin{equation*}
    \overline{H}(x,y)=\gamma ,\qquad \underline{H}(x,y)=1.
  \end{equation*}
\end{lemma}
In \cite{G4}, the following lemma is proved.

\begin{lemma}[{\cite[Proposition~11]{G4}}] \label{prop:htll} 
  Given any system $f$ on a metric space $(X,d)$ let us define
  $d_{n}(x,y)=\min_{0\leq i\leq n} d(f^{i}(x),y)$. Then
  \begin{equation*}
    \underline{H}(x,\tilde{x})=\left( \limsup_{n\rightarrow \infty
    }\frac{-\log d_{n}(x,\tilde{x})}{\log n}\right) ^{-1}
  \end{equation*}
  and
  \begin{equation*}
    \overline{H}(x,\tilde{x})=\left( \liminf_{n\rightarrow \infty
    }\frac{-\log d_{n}(x,\tilde{x})}{\log n}\right) ^{-1}.
  \end{equation*}
\end{lemma}

\begin{proof}[Proof of Proposition~\ref{ooo}]
  We remark that
  \begin{equation*}
    d_{n}(x,0)^{-\beta }\leq S_{n}^{\psi }(x)\leq nd_{n}(x,0)^{-\beta
    }
  \end{equation*}
  and
  \begin{equation*}
    \frac{-\beta \log d_{n}(x,0)}{\log n}\leq \frac{\log S_{n}^{\psi
      }(x)}{\log n }\leq \frac{\log n-\beta \log d_{n}(x,0)}{\log n}.
  \end{equation*}
  By Lemma~\ref{KS...} and\ref{prop:htll} we get
  \begin{align*}
    \liminf_{n \to \infty} \frac{\log S_{n}^{\psi }(x)}{\log n}& \leq
    \liminf_{n \to \infty} \frac{\log n-\beta \log d_{n}(x,0)}{\log n}
    \\ & \leq 1+\frac{\beta }{\gamma}, \\ \limsup_{n \to \infty}
    \frac{\log S_{n}^{\psi }(x)}{\log n}& \geq \beta,
  \end{align*}
  for Lebesgue almost every $x$.
  When $\beta >1,$ and $\frac{1}{\gamma }<1-\frac{1}{\beta }$ we have
  \begin{equation*}
    \liminf_{n \to \infty} \frac{\log S_{n}^{\psi }(x)}{\log n} <
    \limsup_{n \to \infty} \frac{\log S_{n}^{\psi }(x)}{\log
      n},
  \end{equation*}
  for Lebesgue almost every $x$.
\end{proof}

\subsection{Power law mixing examples with oscillating behaviour}
In this section we consider examples of mixing systems also having a strongly
oscillating behaviour. Consider a class of skew products $F\colon[0,1]\times
  S^{1}\rightarrow \lbrack 0,1]\times S^{1}$ defined by
  \begin{equation}
    F_{\theta}(x,t)=(T(x),t+\theta \eta), \label{skewprod}
  \end{equation}
  where
  \begin{equation*}
    T(x)=2x\mod 1,
  \end{equation*}
  and $\eta =1_{[\frac{1}{2},1]}$ is the characteristic function of
  the interval $[\frac{1}{2},1]$.  These maps are piecewise constant
  toral extensions. In these kind of systems, the second coordinate is
  rotated by $\theta $ if the first coordinate belongs to
  $[\frac{1}{2},1]$. We now consider the observable $\tilde{\psi}
  \colon [0,1]\times S^{1}\rightarrow \mathbb{R}$ depending only on
  the second coordinate, an example being $\tilde{\psi} (x,t)=[d
    (t,0)]^{-\beta }$, where $d$ is the distance on $S^{1}$. We have
  the following result.

\begin{prop} \label{prop:oscillating}
  Let $F(x,t)$ be the skew product defined by equation
  \eqref{skewprod}, and suppose that $\tilde{\psi} \colon [0,1]\times
  S^{1}\rightarrow \mathbb{R}$ is given by $\tilde{\psi}(x,t)=[d
    (t,0)]^{-\beta }$ for some $\beta>0$. Then
  \begin{align*}
    \liminf_{n\rightarrow \infty }\frac{\log
      S_{n}^{\tilde{\psi}}(x,t)}{\log n}& \leq 2+\frac{\beta }{\gamma
      (\theta )}, \\ \limsup_{n\rightarrow \infty }\frac{\log
      S_{n}^{\tilde{\psi}}(x,t)}{\log n}& \geq \beta .
  \end{align*}
\end{prop}

\begin{proof}
  We remark that on the system $(F,[0,1]\times S^{1})$ the Lebesgue
  measure is invariant. Let $S_{n}^{\varphi }(x)$ be the Birkhoff sum
  of $\varphi $ for $(T,[0,1]).$ Let $S_{n}^{\psi }(t)$ be the sum of
  the observable $\psi $ on $S^{1}$ as in Proposition~\ref{ooo}. Since
  $\psi $ is positive and $\int \eta \,dm = \frac{1}{2}$, we have by
  the pointwise ergodic theorem applied to $(T,[0,1])$, that for
  a.e.\ $x$
  \begin{equation*}
    S_{\frac{n}{4}}^{\psi }(t)\leq S_{n}^{\tilde{\psi}}(x,t)\leq
    nS_{S_{n}^{\eta}(x,t)}^{\psi }(t)\leq nS_{n}^{\psi }(t)
  \end{equation*}
  holds eventually in $n$. By this and Proposition~\ref{ooo}, we get
  \begin{equation*}
    \beta \leq \limsup_{n\rightarrow \infty }\frac{\log
      S_{\frac{n}{4}}^{\psi }(t)}{\log n}\leq \limsup_{n\rightarrow
      \infty }\frac{\log S_{n}^{\tilde{\psi }}(x,t)}{\log n}
  \end{equation*}
  and
  \begin{equation*}
    \liminf_{n\rightarrow \infty }\frac{\log
      S_{n}^{\tilde{\psi}}(x,t)}{\log n} \leq 2+\frac{\beta }{\gamma
    }.\qedhere
  \end{equation*}
\end{proof}

On the above skew products it is possible to establish power law
bounds for the rate of decay of correlations on Lipschitz
observables. The power law exponent depend on the Diophantine type of
the translation angle $\theta$.  In \cite[Lemma~11 and
  Section~5]{GSR}, the following is proved.

\begin{prop} \label{prop:poldecay}
  Let $p$ be the exponent of power law decay with respect to Lipschitz
  observables defined by
  \begin{equation*}
    p=\liminf_{n\rightarrow \infty }\frac{-\log \Theta (n)}{\log n},
  \end{equation*}
	where $\Theta(n)$ is the correlation decay rate function.
  For the map defined by \eqref{skewprod} the exponent $p$ satisfies
  \begin{equation*}
    \frac{1}{2\gamma (\theta )}\leq p\leq \frac{6}{\max (2,\gamma
      (\theta ))-2}.
  \end{equation*}
\end{prop}

Combining Propositions~\ref{prop:oscillating} and \ref{prop:poldecay},
we have found systems with polynomial decay of correlations, for which
Birkhoff sums show the same strongly oscillating behaviour as in the
case of rotations. This result establishes the first result stated in
Theorem~\ref{the:thereareoscillations}.

\subsection{Mixing systems with slowly increasing time averages\label{slow}}

In this section we see examples where $\limsup_{n\rightarrow \infty }
\frac{\log S_{n}^{\psi }(x)}{\log n}$ is bounded from above by the
arithmetical properties of the system and can have a very slow increase. To construct these
examples we consider two dimensional rotations with suitable angles.

Let us consider a rotation $f_{\mathbf{\theta }}$ of the torus
$\mathbb{T}^{2}\cong \mathbb{R}^{2}/\mathbb{Z}^{2}$ by an angle with
components $(\theta, \theta ^{\prime })$. Suppose that $(\gamma
,\gamma ^{\prime })$ are respectively the types of $\theta $ and
$\theta ^{\prime }$.  Denote by $q_{n}$ and $q_{n}^{\prime }$ the
partial convergent denominators of $\theta $ and $\theta ^{\prime }$.
Let us consider $\xi >1$ and let $Y_{\xi }\subset \mathbb{R}^{2}$ be
the class of couples of irrationals $(\theta ,\theta^{\prime})$ given
by the following conditions on their convergents to be satisfied
eventually
\begin{align*}
  q_{n}^{\prime } & \geq q_{n}^{\xi },\\
  q_{n+1} & \geq q_{n}^{\prime }{}^{\xi }.
\end{align*}
The set $Y_{\xi }$ is uncountable, dense in $[0,1]\times \lbrack 0,1]$
  and there are points in $Y_{\xi }$ having finite Diophantine type
  coordinates.  If we take angles in $Y_{\xi }$ the lower hitting
  time indicator is bounded from below by $\xi $. In
  \cite[Section~6]{GSR} the following is proved.

\begin{prop} \label{exest}
  Consider the class of skew products $F\colon[0,1]\times
  S^{1}\times S^{1}\rightarrow \lbrack 0,1]\times S^{1}\times S^{1}$
  defined by
  \[
    F(x,t)=(T(x),t+\mathbf{\theta } \eta (x)),
  \]
  where $\mathbf{\theta} \in Y_{\xi }$,
  \[
    T(x)=2x\mod 1,
  \]
  and $\eta =1_{[\frac{1}{2},1]}$ is the characteristic function of
  the interval $[\frac{1}{2},1]$. For each $y\in \lbrack 0,1]\times
  S^{1}\times S^{1}$ it holds
  \[
  \underline{H}(x,y)\geq \max (3,\xi )
  \]
  for a.e.\ $x$. Furthermore, there are infinitely many
  $\mathbf{\theta} \in Y_{4}$ such that $F$ is polynomially mixing
  with respect to Lipschitz observables.
\end{prop}

In \cite{G2}, the following is proved.

\begin{prop}[{\cite[Theorem~11]{G2}}]
  Consider a dynamical system $(f,X,\mu)$ where $X$ is a metric space
  and $f$ a Borel map. Let $\tilde{x}\in X$ and let us consider the
  observable $\phi (x)=d(x,\tilde{x})^{-k}$, where $k \geq 0$. Let
  $S_{n}^{\phi }(x)$ be the usual Birkhoff sum. Then it holds for
  $\mu$-a.e.\ $x\in X$
  \begin{equation*}
    \limsup_{n\rightarrow \infty }\frac{\log S_{n}^{\phi }(x)}{\log
      n}\leq \frac{ k}{\underline{H}(x,\tilde{x})}+1.
  \end{equation*}
\end{prop}

By this we easily get the following proposition.

\begin{prop}
  Consider the system as described in Proposition~\ref{exest}. Then
  \begin{equation*}
    \limsup_{n\rightarrow \infty }\frac{\log S_{n}^{\phi }(x)}{\log
      n}\leq \frac{k}{\max (3,\xi )}+1.
  \end{equation*}
\end{prop}

Then in these kind of systems when $\xi $ is large, the growth of
Birkhoff sums can be slow even for large $k$ (while the dimension of
the invariant measure is $3$ for each choice of $\mathbf{\theta }$).
We remark that already when $\xi =4$, if $k>12$, then we can find systems having power law decay of correlations and for which
\begin{equation*}
\limsup_{n\rightarrow \infty }\frac{\log S_{n}(x)}{\log n}<\frac{k}{d_{\mu
}(\tilde{x})}
\end{equation*}
as stated in the second part of Theorem~\ref{the:thereareoscillations}.

\section{Borel--Cantelli results for infinite systems modelled by Young towers} \label{sec:youngtowers}
In this final section we consider application of
Theorem~\ref{thm:infiniteSBC} (Section~\ref{BC}) to infinite systems
$(f,X,\mu)$ modelled by Young towers. To keep the exposition simple,
we focus on dynamical Borel--Cantelli results for sequences of nested
balls $(B_n)$, i.e.\ $B_k\subset B_l$ if $k>l$. Based on the results
of Section~\ref{sec:mainresults} this theory can be easily
generalised, e.g.\ in finding almost sure bounds for the growth of
$M_n$ and $S_n$.  For dynamical Borel--Cantelli results for systems
preserving a probability measure, see \cite{GNO}.

We describe the Young tower as follows (compare to the suspension
construction outlined in Section~\ref{extremes}). Consider a measure
preserving (one-dimensional) system $(f,X,\mu)$. Suppose there exists
a subinterval $\Lambda \subset X$, and a countable partition
$\{\Lambda_{i}\}_{i\in\mathbb{N}}$ of $\Lambda$ into sub-intervals,
together with a function $R\colon \Lambda\to\mathbb{N}$ defined by
$R|_{\Lambda_i}=R_i$ if $f^{R_i} \colon \Lambda_i \to \Lambda$ is a
bijection. The set $\Delta$ (denoting the Young tower) is defined by
\[
\Delta = \bigcup_{k=1}^{\infty}\bigcup_{\ell=0}^{R_k}\Delta_{k,\ell},
\]
where $\Delta_{k,\ell}:= \Lambda_k \times \{\ell\}$ is a subset of
\emph{level} $\ell$. The set $\Delta_0 = \bigcup_k \Delta_{k,0}$ is
identified with $\Lambda$.
The tower map $F\colon \Delta\to\Delta$ is given by
\begin{equation*} 
  F(x,\ell)= \left\{ \begin{array}{ll} (x,\ell+1) & \textrm{if
      $x\in\Lambda_i$ and $\ell<R_i-1$},\\ (f^{R_i}(x),0) & \textrm{if
      $x\in\Lambda_i$ and $\ell=R_i-1$}. \end{array} \right.
\end{equation*}
From this construction, the base map $\hat{F}:=F^R \colon
\Lambda\to\Lambda$ is a Gibbs--Markov map, and admits an ergodic
absolutely continuous invariant measure $\hat{\nu}$, uniformly
equivalent to Lebesgue measure, in the sense that there exists a $C>0$
such that
\[
\frac{1}{C}\leq \frac{d\hat{\nu}}{dm}\leq C.
\]
The measure $\hat{\nu}$ lifts to an invariant measure $\nu$ for $F$ on
$\Delta$ by defining
\[
\nu(A)=\hat{\nu}(F^{-\ell}(A)),
\]
for $A\subset\Delta_{k,\ell}$. In particular for a general set
$A\subset\Delta$ we have
\[
\nu(A)=\sum_{\ell}\hat{\nu}\left(\{R>\ell\}\cap
F^{-\ell}{A}\right).
\]
In the case where $\left< R \right>=\int_{\Lambda}R\,
d\hat{\nu}<\infty$, we can normalise $\nu$ to a probability
measure. The system $(F,\Delta,\nu)$ forms a Markov extension of
$(f,X,\mu)$. In particular there exists a (surjective) factor map
$\pi \colon \Delta \to X$ satisfying $\pi\circ F=f\circ
\pi$, and for $(x,\ell)\in\Delta$ we have explicitly
$\pi(x,\ell)=f^{\ell}(x)$. The measure $\mu=\pi_{*}\nu$ is $f$-invariant, and
is a probability measure in the case where $\nu$ is a probability
measure. Here, for $A\subset X$ we have
$\pi_*\nu(A)=\nu(\pi^{-1}(A)).$ Under further regularity assumptions
on the system $(F,\Delta, \nu)$ the measures $\nu$ and $\mu$ can be shown
to be absolutely continuous with respect to $m$, see \cite{Young, Young2}.

In the present situation we consider the case $\left<R\right>=\infty$,
and hence $\nu(\Delta)=\infty$. Consider now
$(f,X,\mu)$ ergodic, conservative and that $\mu$ is
absolutely continuous with respect to Lebesgue measure. We
can again build the Young tower $(F,\Delta,\nu)$ over
$(f,X,\mu)$. However, unlike the case
$\left<R\right><\infty$, the measure $\mu':=\pi_{*}\nu$ need not be
$\sigma$-finite unless additional conditions are satisfied on the
return time function $R$, see \cite{BNT}. We assume that $\pi_{*}\nu$
is $\sigma$-finite.
We have the following result.

\begin{thm}\label{thm:tower}
  Suppose that $(f,X,\mu)$ is an ergodic, conservative measure
  preserving map on an interval $X\subset\mathbb{R}$, and furthermore
  $(f,X,\mu)$ is modelled by a (one-dimensional) Young tower $(F,\Delta,\nu)$, with
  return time function satisfying $\hat\nu\{R=n\}\sim n^{-\beta-1}$,
  for some $\beta\in(0,1)$, and that the measure $\mu=\pi_{*}\nu$ is
  $\sigma$-finite. Then there exists a set $X'$ with
  $\mu(X\setminus X')=0$ with the following property. If $\{B_n\}$ and is a nested
  sequence of intervals with $\cap_n B_n=\{\tilde{x}\}$, $\tilde{x}\in X'$,
and $\sum_{k=1}^\infty \mu(B_{k^{\alpha+\varepsilon_1}}) = \infty$
(for some $\varepsilon_1>0$), then for all $\epsilon\in(0,\varepsilon_1]$
  \begin{equation}\label{eq.quant-tower}
    \liminf_{n\to\infty}\frac{\sum_{k=1}^{n} 1_{B_k} (f^k
      (x))}{\sum_{k=1}^{n^{\frac{1}{\alpha+\varepsilon}} }
      \mu(B_{k^{\alpha+\varepsilon}})} \geq 1
  \end{equation}
  for $\mu$-a.e.\ $x$. Here $\alpha=1/\beta$.
\end{thm}

Before proving Theorem~\ref{thm:tower} we make the following remarks.
First of all, the set $X'$ consists of points where the density of $\mu$ is
bounded away from $\{0,\infty\}$, and where $\pi^{-1}(\tilde{x})$ is non-empty and
consists only of interior points within each $\Delta_{k,\ell}$.
The partition sets $\Delta_{k,\ell}$ here are identified with intervals in $X$ under
the projection $\pi$.
For more general Young tower constructions, even in one dimension, it is possible
for the sets $\Delta_{k,\ell}$
to have more general (Cantor set) geometries. We do not consider these latter situations.

In the case $\left<R\right><\infty$, a corresponding theorem is
established in \cite{GNO}, where a \emph{dense Borel--Cantelli}
property is achieved (i.e.\ a lower bound on the $\liminf$ analogous
to equation \eqref{eq.quant-tower}). Their result does not require
the assumption $\hat\nu\{R=n\}\sim n^{-\beta-1}$. In the
case $\left<R\right>=\infty$, the asymptotics of $R$ play a role in
the statement on the Borel--Cantelli result via the constant $\alpha=1/\beta$.

Relative to equation \eqref{eq.quant} in Theorem~\ref{thm:infiniteSBC},
equation \eqref{eq.quant-tower} only gives an almost sure lower bound
on $\sum_{k=1}^{n} 1_{B_k} (f^k(x))$. (This lower bound can be further refined
following the remarks given immediately after Theorem~\ref{thm:infiniteSBC}).
For upper bounds, the issue arising here is that
$\pi^{-1}(B_n)$ can have non-empty intersection with a countably infinite number
of sets of the form $\Delta_{k,\ell}$, and dynamical Borel--Cantelli results do not in general
carry over to countable unions of (shrinking target) sets.
However, in certain situations it is
possible to re-arrange the Young Tower so that
$\pi^{-1}(B_n)$ is contained in a finite number of $\Delta_{k,\ell}$,
thus bypassing the problem. Indeed it is not difficult to show that
such a tower can be constructed for the intermittent map family
of maps \eqref{eq.LSVmap}
given in Section~\ref{sec.intermittent}. This would apply in the
case where the sets $(B_n)$ do not accumulate at the neutral fixed
point $\tilde{x}=0$. To see how to construct such a tower, see the proof of Theorem~\ref{thm:maximaint}
and Corollary~\ref{cor.BC}.  We state the following result.

\begin{cor}\label{cor.towerSBC}
  Under the assumptions of Theorem~\ref{thm:tower}, suppose that for
  all $n$ sufficiently large we have $\pi^{-1}(B_n)$ contained in a
  finite number of $\Delta_{k,\ell}$. Then for $\mu$-a.e.\ $x\in X$,
 and all $\epsilon\in(0,\varepsilon_1]$ we have eventually in $n$ (as
$n\to\infty)$
  \begin{equation*}
    \sum_{k=1}^{n^{\frac{1}{\alpha+\varepsilon}}
    }\mu(B_{k^{\alpha+\varepsilon}}) \leq \sum_{k=1}^{n} 1_{B_n}
    (f^k(x)) \leq \sum_{k=1}^{n^{\frac{1}{\alpha-\varepsilon}}
    }\mu(B_{k^{\alpha-\varepsilon}}).
  \end{equation*}
\end{cor}

We now prove Theorem~\ref{thm:tower}, and Corollary~\ref{cor.towerSBC}.

\begin{proof}[Proof of Theorem~\ref{thm:tower}]
  For $n>0$, consider the sets $\pi^{-1}(B_n)$ on the tower
  $\Delta$.  In general, $\pi^{-1}(\tilde{x})\subset\Delta$ is
  a countably infinite, and for each $\Delta_{k,\ell}$, the set
  $\pi^{-1}(\tilde{x})\cap\Delta_{k,\ell}$ contains at most one point
  since $F^\ell$ is one-to-one on $\Delta_{k,\ell}$ for $\ell \leq R_k$. Let
  $\tilde{x}_{k,\ell}$ be defined by
  $\tilde{x}_{k,\ell}=\pi^{-1}(\tilde{x})\cap\Delta_{k,\ell}$ whenever
  this intersection is nonempty.
  We let $\hat{x}_{k,\ell}\in\Delta_{k,0}$ be the point on the base
  such that $F^{\ell}(\hat{x}_{k,\ell})=\tilde{x}_{k,\ell}$.

  In the following, it suffices to work with a given partition element
  $\Delta_{k,\ell}$.  Let $B_n'=\pi^{-1}(B_n)\cap\Delta_{k,\ell},$ and
  $D_n'=F^{-\ell}B_n'\subset\Delta_0.$ First of all, consider
  $x\in\Delta_0$, and let
  \[
    \hat{S}(n,x)=\sharp\{\, j<n :F^j(x,0)\in B_n' \,\}.
  \]
  We now apply Theorem~\ref{thm:infiniteSBC} to the sequence $D_{n}'$
  to obtain (eventually in $n$)
  \[
  \sum_{k = 1}^{n^{\frac{1}{\alpha + \varepsilon}}} \hat{\nu}
  (D_{k^{\alpha+\varepsilon}}') \leq \sum_{k=1}^{n} 1_{D_k'} (F^k(x))
  \leq \sum_{k=1}^{n^{\frac{1}{\alpha-\varepsilon}} }
  \hat{\nu}(D_{k^{\alpha-\varepsilon}}'),
  \]
where we recall $\hat\nu$ is the $\hat{F}$-invariant measure.
  We translate this immediately to a statement on returns to $B_n'$ by
  noting the following: we pass through $B_n'$ at most once between
  successive returns to the base of the tower, and by definition of
  $\nu$, we have $\nu(A)=\hat{\nu}(F^{-\ell}(A))$ in the case
  $A\subset\Delta_{k,\ell}$. Hence there exists $n_0(\Delta_{k,\ell})$ such
that for all $n\geq n_0$
  \begin{equation} \label{eq.tower-nu-SBC}
	\sum_{k = 1}^{n^{\frac{1}{\alpha + \varepsilon}}} \nu
      (B_{k^{\alpha+\varepsilon}}')
			\leq \sum_{k=1}^{n} 1_{B_k'} (F^k(x))
	\leq \sum_{k=1}^{n^{\frac{1}{\alpha-\varepsilon}} }
    \nu(B_{k^{\alpha-\varepsilon}}').
	 \end{equation}

  We now study the preimages $\pi^{-1}(B_n)$ on the whole tower, using
 the assumption $\mu=\pi_{*}\nu$ is $\sigma$-finite.  We might expect equation
  \eqref{eq.tower-nu-SBC} to apply to $(f, X,\mu)$ by summing over all
  relevant $B_n'$ that contain a $\tilde{x}_{k,\ell}\in
  \pi^{-1}(\tilde{x})$. However, the eventual bounds in \eqref{eq.tower-nu-SBC} depend
on $\Delta_{k,\ell}$, and since we have a countable infinite number of such sets
we cannot get uniform asymptotic estimates in $n$.
In the current scenario, and as considered in \cite{GNO} (in the finite measure
  case), it suffices to work with a truncated tower
  \[
  \Delta_{N(\delta)} := \cup \{\, \Delta_{k,\ell} : k\leq N(\delta),
  \ell < \min \{ N(\delta), R_k \} \,\},
  \]
  and for any $\delta>0$, choose $N(\delta)>0$ such that
  \[
  \mu(B_n \cap \pi (\Delta_{N(\delta)})) \geq (1 - \delta) \mu(B_n).
  \]
  The existence of $N(\delta)>0$ follows from the fact that
each $\Delta_{k,\ell}$ is identified with a subinterval in $X$ under
the projection $\pi$. Each $\pi^{-1}(\tilde{x})$ is an interior point of
$\Delta_{k,\ell}$, and since $B_n$ is a sequence of balls with $\mu(B_n)\to 0$,
we can find $\Lambda_k\subset\Lambda_0$ with $f^{\ell}(\Lambda_k)\supset B_n$
for all $n$ sufficiently large. For
  every $\tilde{x}_{k,\ell}\in\Delta_{N(\delta)}$ we can repeat the
  derivation of equation \eqref{eq.tower-nu-SBC}. We proceed as
  follows.

  We define $\hat{S}_{\Delta_{N(\delta)}} (n,x)$ by
  \[
  \hat{S}_{\Delta_N}(n,x)=\sharp\{\, j<n:F^j(x,0)\in
  \cup_{B_j'\subset\Delta_{N(\delta)}} B_j' \,\},
  \]
  and similarly for $\hat{S}_\Delta (n,x)$.
  From \eqref{eq.tower-nu-SBC} we obtain that
  \[
  \hat{S}_{\Delta_{N(\delta)}} (n,x) \geq (1-\delta) \sum_{k =
    1}^{n^{\frac{1}{\alpha + \varepsilon}}} \nu (\pi^{-1} B_{k^{\alpha
      + \varepsilon}} ),
  \]
  if $n$ is large enough. Here we used that $N(\delta)$ is finite, and
  hence there are a finite number of components we are considering for
  each $\delta$ fixed. The asymptotic properties of the
  Borel--Cantelli results thus carry over.

  We have $\hat{S}_\Delta (n,x) \geq \hat{S}_{\Delta_{N(\delta)}}
  (n,x)$ and so for sufficiently large $n$
  \[
  \hat{S}_{\Delta} (n,x) \geq (1-\delta) \sum_{k =
    1}^{n^{\frac{1}{\alpha+\varepsilon}}} \nu (\pi^{-1} \left(B_{k^{\alpha +
      \varepsilon}}\right) ),
  \]
 Since $\mu = \nu \circ \pi^{-1}$, we obtain
  \[
  S_n (x) \geq (1-\delta) \sum_{k =
    1}^{n^{\frac{1}{\alpha+\varepsilon}}} \mu (B_{k^{\alpha +
      \varepsilon}}),
  \]
 where $S_n (x) = \sharp \{\, j \leq n : f^j (x) \in B_j \,\}.$
  Since $\delta >0$ is arbitrary, this proves the theorem.
\end{proof}

\begin{proof}[Proof of Corollary~\ref{cor.towerSBC}]
  Relative to the proof of Theorem~\ref{thm:tower} we now get an
  almost sure upper bound for $\hat{S}_\Delta (n,x)$.  We have
  \[
  \mu (B_n \cap \pi(\Delta_{N (\delta)})) \leq \mu (B_n).
  \]
  In the same way as above, we then obtain that
  \[
  \hat{S}_{\Delta_{N(\delta)}} (n,x) \leq \sum_{k =
    1}^{n^{\frac{1}{\alpha - \varepsilon}}} \nu (\pi^{-1}
  (B_{k^{\alpha - \varepsilon}}) ).
  \]
  With inequalities in this direction, we cannot in general estimate
  $\hat{S}_\Delta (n,x)$ by $\hat{S}_{\Delta_{N(\delta)}} (n,x)$.
  Let
  \[
  \hat{T}_{N(\delta)} (n,x) = \sharp \{\, j \leq n : F^j (x,0) \not \in
  \Delta_{N(\delta)} \cap \pi^{-1} (B_n) \,\}
  \]
  count the visits to the ``tail'' $(\Delta \setminus
  \Delta_{N(\delta)}) \cap \pi^{-1} (B_n)$.  We then have
  \[
  \hat{S}_\Delta (n,x) \leq \hat{S}_{\Delta_{N(\delta)}} (n,x) +
  \hat{T}_{N(\delta)} (n,x),
  \]
  However, since by assumption $\pi^{-1}(B_n)$ is contained in a
  finite number of $\Delta_{k,\ell}$, there exists $N'$, such that
  $\Delta_{k,\ell}\cap\pi^{-1}(B_n)=\emptyset$ for all $k,\ell\geq
  N'$.  We then just choose $N(\delta)>N'$ so that
  $\hat{T}_{N(\delta)} (n,x) = 0$. The conclusion now follows.
\end{proof}

\section{Appendix}\label{sec.appendix}

\begin{proof}[Proof of Lemma~\ref{slem.gamma}]

We have $\mathcal{P} = \{X_i\}_{i \in\mathcal{I}}$, with $\mathcal{I}$ an index set
$\subset\mathbb{N}$. There is a one-to-one
correspondence between elements of the partition $\mathcal{P}_{n}$ and
sequences $(i_0,i_1,\ldots,i_{n-1}) \in \mathcal{I}^n$, which is characterised
by the property that for $X_{i_0,i_1,\ldots,i_{n-1}} \in
\mathcal{P}_n$ we have
\[
f^k (X_{i_0,i_1,\ldots,i_{n-1}}) \subset X_{i_k} \qquad \text{for }
k = 0, 1, \ldots, n - 1.
\]
Note that some sequences $(i_0,i_1,\ldots,i_{n-1})$ are not admissible, unless
we assume $f(X_i)=X$ (since we only require big images from assumption \ref{G-M1}). However, this fact is
of no consequence to what follows.
We may assume that the diameter of the interval $X$ is one. The
bounded distortion assumption~\ref{G-M3}
then implies that
\[
|X_{i_0,i_1,\ldots,i_{n-1}}| \leq e^{C \tau^n} |X_{i_0, i_1 \ldots
i_{n-2}}| \cdot |X_{i_{n-1}}|.
\]

The measure $\mu$ is absolutely continuous with respect to the
Lebesgue measure, and there exists a constant $c$ such that the
density of $\mu$ is bounded by $c$ and bounded away from zero by
$c^{-1}$. Consider
\[
Q_n := |\{\, x \in X : \phi (f^j(x)) < \gamma_n \text{ for all }
j < n \,\}|.
\]
Since the density of $\mu$ is bounded, we have $P_n \leq c Q_n$.

We proceed by induction to estimate $Q_n$. Let $n$ be fixed and define
\[
Q_{n,k} = |\{\, x \in X : \phi (f^j(x)) < \gamma_n \text{ for all }
j < k \,\}|.
\]
In particular, we have $Q_n = Q_{n,n}$.

Clearly, we have
\[
Q_{n,1} = \mu \{\, x \in X : \phi (x) < \gamma_n \,\} = 1 - \mu \{\, x
\in X : \phi (x) \geq \gamma_n \,\} \leq (1 - D_0 \gamma_n^{- \beta}),
\]
for some constant $D_0$. Suppose that we have
\[
Q_{n,k-1} \leq e^{\sum_{j=0}^{k-1} C \tau^j} (1 - D_0 \gamma_n^{-
  \beta})^{k-1},
\]
for some $k$. Then
\begin{align*}
  Q_{n,k} &= \sum_{\substack{i_0,i_1,\ldots,i_{k-1} \\ \phi(i_j) <
      \gamma_n }} |X_{i_0,i_1,\ldots,i_{k-1}}| \\ & \leq
  \sum_{\substack{i_0,i_1,\ldots,i_{k-1} \\ \phi(X_{i_j}) < \gamma_n
  }} e^{C \tau^k} |X_{i_0, i_1, \ldots, i_{k-2}}| \cdot |X_{i_{k-1}}|
  \\ & \leq e^{C \tau^k} Q_{n,k-1} \sum_{i_{k-1} : \phi (X_{i_{k-1}})
    < \gamma_n} |X_{i_k}| = e^{C \tau^k} Q_{n,k-1} Q_{n,1} \\ &\leq
  e^{\sum_{j=0}^k C \tau^j} (1 - D_0 \gamma_n^{-\beta})^k.
\end{align*}
In particular, we have
\[
Q_n = Q_{n,n} \leq e^{\frac{C}{1 - \tau}} (1 - D_0
\gamma_n^{-\beta})^n
\]
and
\[
P_n \leq c Q_n \leq c e^{\frac{C}{1 - \tau}} (1 - D_0
\gamma_n^{-\beta})^n,
\]
which finishes the proof.
\end{proof}


\begin{thebibliography}{99}
\bibitem{Aar2} J. Aaronson, \emph{On the ergodic theory of
  non-integrable functions and infinite measure spaces}, Israel
  J. Math., {\bf 27}, (2), (1977), 163--173.

\bibitem{Aaronson} J. Aaronson, \emph{An Introduction to Infinite
  Ergodic Theory}, Mathematical Surveys and Monographs, Volume 50,
  American Mathematical Society, 1997.

\bibitem{AaronsonDenker} J. Aaronson, M. Denker, \emph{Upper bounds
  for ergodic sums of infinite measure preserving transformations},
  Trans. Amer. Math. Soc. {\bf 319}, (1), (1990), 101--138.

\bibitem{AGW90} R. Arratia, L. Gordon and M. Waterman, \emph{The
  Erd\H{o}s--R\'{e}nyi law in distribution, for coin tossing and
  sequence matching}, Ann. Stat., {\bf 18}, (1990), 539--570.

\bibitem{baladi} V. Baladi, \emph{Decay of correlations},
  Proc. Symposia in Pure Math., Vol. 69, AMS (2001), pp 297--325.

\bibitem{Bal01} N. Balakrishnan, M. V. Koutras, \emph{Runs and Scans
  with Applications}, Wiley, New York, 2001.

\bibitem{Bar92} A. D. Barbour, L. Holst, S. Janson, \emph{Poisson
  Approximation}, Clarendon Press, Oxford, 1992.

\bibitem{Bat48} G. I. Bateman. \emph{On the power function of the
  longest run as a test for randomness in a sequence of
  alternatives}, Biometrika {\bf 35}, (1948), 97--112.

\bibitem{BY} M. Benedicks, L. S. Young, \emph{Markov extensions and
  decay of correlations for certain H\'{e}non maps}, Ast\'erisque {\bf
  261}, (xi), (2000), 13--56.

\bibitem{BNT} H. Bruin, M. Nicol, D. Terhesiu, \emph{On Young towers
  associated with infinite measure preserving transformations},
  Stoch. Dyn., {\bf 9}, (4), (2009), 635--655.

\bibitem{BL14} Y. Bugeaud, L. Liao, \emph{Uniform Diophantine
  approximation related to $b$-ary and $\beta$-expansions},
  Ergodic Theory Dynam. Systems, {\bf 36}, (1), (2016), 1--22,

\bibitem{CN17} M. Carney, M. Nicol, \emph{Dynamical Borel--Cantelli
  lemmas and the rate of growth of Birkhoff sums of non-integrable
  obersvervations on chaotic dynamical systems}, Nonlinearity,  {\bf 30}, (7), (2017),
 2854--2870.



\bibitem{CFZ18} H. Cui, L. Fang, Y. Zhang, \emph{A note on the run
  length function for intermittency maps},  Arxiv:
  https://arxiv.org/abs/1806.09363, 2018.

\bibitem{denzab07} M. Denker, Z. Kabluchko, \emph{An
  Erd\H{o}s--R\'{e}nyi law for mixing processes},
  Probab. Math. Satist., {\bf 27}, (2007), 139--149.

\bibitem{dennic13}M. Denker, M. Nicol, \emph{Erd\H{o}s--R\'{e}nyi
  laws for dynamical systems}, J. London Maths. Soc., {\bf 87}, (2),(2013),
  497--508.

\bibitem{Embrechts} P. Embrechts, C. Kl\"{u}pperlberg, and T. Mikosch,
  \emph{Modelling extremal events. For insurance and finance},
  Applications of Mathematics (New York), 33, Springer-Verlag, Berlin,
  1997.

\bibitem{Eroren70} P. Erd\H{o}s, A. R\'{e}nyi, \emph{On a new law
  of large numbers}, J. Anal. Math. {\bf 23}, (1970), 103--111.

\bibitem{FW12} A. H. Fan, B. W. Wang, \emph{On the lengths of basic
  intervals in beta expansions}, Nonlinearity, {\bf 25}, (5), (2012),
  1329--1343.

\bibitem{FFT1} J. Freitas, A. Freitas, M. Todd. \emph{Hitting times and
  extreme value theory}, Probab. Theory Related Fields,
  \textbf{147}, (3), (2010), 675--710.

\bibitem{FFT2} A. C. M. Freitas, J. M. Freitas, M. Todd, \emph{Extreme
  value laws in dynamical systems for non-smooth observations},
  J. Stat. Phys., \textbf{42}, (1), (2011) 108--126.

\bibitem{Galambos} J. Galambos, \emph{The Asymptotic Theory of Extreme
  Order Statistics}, John Wiley and Sons, 1978.

\bibitem{G2} S. Galatolo, \emph{Hitting time and dimension in axiom A
  systems, generic interval exchanges and an application to Birkoff
  sums}, J.  Stat. Phys., {\bf 123}, (1), (2006), 111--124.

\bibitem{G3} S. Galatolo, \emph{Dimension and hitting time in rapidly
  mixing systems}, Math. Res. Lett., {\bf 14}, (5), (2007), 797--805.

\bibitem{G} S. Galatolo, \emph{Hitting time in regular sets and
  logarithm law for rapidly mixing dynamical systems},
  Proc. Amer. Math. Soc., {\bf 138}, (7), (2010), 2477--2487.

\bibitem{GK} S. Galatolo, D. H. Kim, \emph{The dynamical
  Borel--Cantelli lemma and the waiting time problems},
  Indag. Math. (N.S.), {\bf 18}, (3), (2007), 421--434.

\bibitem{GKP} S. Galatolo, D. H. Kim, K. Koh Park, \emph{The recurrence time for ergodic systems with infinite invariant measures}, Nonlinearity, {\bf 19},
(11), (2006), 2567–-2580.

\bibitem{GN} S. Galatolo, I. Nisoli, \emph{Shrinking targets in fast
  mixing flows and the geodesic flow on negatively curved manifolds},
  Nonlinearity, {\bf 24}, (11), (2011), 3099--3113.

\bibitem{GP} S. Galatolo, M. J. Pacifico, \emph{Decay of correlations
  for maps with uniformly contracting fibers and logarithm law for
  singular hyperbolic attractors}, Math. Z., {\bf 276}, (3--4), (2014),
  1001--1048.

\bibitem{GP2} S. Galatolo, M. J. Pacifico, \emph{Lorenz-like flows:
  exponential decay of correlations for the Poincar\'{e} map,
  logarithm law, quantitative recurrence}, Ergodic Theory Dynam. Systems, {\bf 30}, (6), (2010), 1703--1737.

\bibitem{G4} S. Galatolo, P. Peterlongo, \emph{Long hitting time, slow
  decay of correlations and arithmetical properties}, Discrete
  Contin. Dyn. Syst., {\bf 27}, (1), (2010), 185--204.

\bibitem{GSR} S. Galatolo, J. Rousseau, B. Saussol, \emph{Skew
  products, quantitative recurrence, shrinking targets and decay of
  correlations}, Ergodic Theory Dynam. Systems, {\bf 35}, (6), (2015),
  1814--1845.

\bibitem{Gouezel} S. Gou\"{e}zel, \emph{A Borel--Cantelli lemma for
  intermittent interval maps}, Nonlinearity, {\bf 20}, (6), (2007),
  1491--1497.

\bibitem{Gri93} J. Grigull, \emph{Gro\ss e Abweichungen und
  Fluktuationen f\"ur Gleichgewichtsma\ss e rationaler Abbildungen},
  Dissertation, Georg-August-Universit\"at G\"ottingen, 1993.

\bibitem{GNO} C. Gupta, M. Nicol, W.~Ott, \emph{A Borel--Cantelli lemma
  for non-uniformly expanding dynamical systems}, Nonlinearity,
  \textbf{23}, (8), (2010), 1991--2008.

\bibitem{HNPV} N. Haydn, M. Nicol, T. Persson, S. Vaienti, \emph{A
  note on Borel--Cantelli lemmas for non-uniformly hyperbolic
  dynamical systems}, Ergodic Theory Dynam. Systems, {\bf 33}, (2), (2013),
  475--498.

\bibitem{HNT} M. Holland, M. Nicol, A. T\"{o}r\"{o}k, \emph{Almost
  sure convergence of maxima for chaotic dynamical systems},
  Stochastic Process. Appl. {\bf 126}, (10), (2016), 3145--3170.

\bibitem{HVRSB} M. Holland, R. Vitolo, P. Rabassa, A. E. Sterk,
  H. Broer, \emph{Extreme value laws in dynamical systems under
    physical observables}, Phys. D., {\bf 241}, (2012), 497--513.

\bibitem{Kim} D. H. Kim, \emph{The dynamical Borel--Cantelli lemma for
  interval maps}, Discrete Contin. Dyn. Syst., {\bf 17}, (4), (2007),
  891--900.

\bibitem{kimseo} D. H. Kim, B. K. Seo, \emph{The waiting time for
  irrational rotations}, Nonlinearity, {\bf 16}, (2003), 1861--1868.

\bibitem{LeMu} M. Lenci, S. Munday, \emph{Pointwise convergence of Birkhoff averages for global observables},
 Chaos, {\bf 28}, (8), (2018), 083--111.


\bibitem{LSV} C. Liverani, B. Saussol, S. Vaienti, \emph{A
  probabilistic approach to intermittency}, Ergodic Theory Dynam. Systems,
  {\bf 19}, (3), (1999), 671--685.

\bibitem{V.et.al} V. Lucarini, D. Faranda, A. Freitas,
  J. Freitas, M. Holland, T. Kuna, M. Nicol, M. Todd,
  S. Vaienti, \emph{Extremes and Recurrence in Dynamical Systems},
  Pure and Applied Mathematics: A Wiley Series of Texts, Monographs,
  and Tracts, 2016.

\bibitem{MP} P. Manneville, Y. Pomeau, \emph{Intermittent transition to
  turbulence in dissipative dynamical systems}, Comm. Math. Phys.,
  {\bf 74}, (2), (1980), 189--197.

\bibitem{Melbourne-Terhesiu} I. Melbourne, D. Terhesiu, \emph{Operator
  renewal theory and mixing rates for dynamical systems with infinite
  measure}, Invent. Math., {\bf 189}, (1), (2012), 61--110.

\bibitem{pes} Y. B. Pesin, \emph{Dimension theory in dynamical
  systems}, Contemporary views and applications, Chicago Lectures in
  Mathematics, University of Chicago Press, Chicago, IL, 1997.

\bibitem{deMo1738} A. de Moivre, \emph{The Doctrine of Chances}, H.
  Woodfall, London, 1738.

\bibitem{Mus00} M. Muselli, \emph{New improved bounds for reliability
  of consecutive-k-out-of-n:F systems}, J. Appl. Probab., {\bf 37}, (2000),
  1164--1170.

\bibitem{Tonyuzhao} X. Tong, Y. Yu, Y. Zhao, \emph{On the maximal
  length of consecutive zero digits of $\beta$-expensions},
  Int. J. Number Theory, {\bf 12}, (3), (2016), 625--633.

\bibitem{Cori} C. Ulcigrai, \emph{Mixing of asymmetric logarithmic
  suspension flows over interval exchange transformations},
  Ergodic Theory Dynam. Systems, {\bf 27}, (3), (2007), 991--1035.

\bibitem{Cori2} C. Ulcigrai, \emph{Absence of mixing in
  area-preserving flows on surfaces}, Ann. of Math., {\bf 173}, (3), (2011),
  1743--1778.

\bibitem{Young} L. S. Young, \emph{Statistical properties of dynamical
  systems with some hyperbolicity}, Ann. of Math., {\bf 147}, (3), (1998),
585--650.

\bibitem{Young2} L. S. Young, \emph{Recurrence times and rates of mixing},
Israel J Math., {\bf 110}, (1999), 153--188.

\bibitem{Zwm} R. Zweim\"{u}ller,
\emph{Invariant measures for general(ized) induced transformations},
 Proc. Amer. Math. Soc., {\bf 133}, (2005), 2283--2295.
\end{thebibliography}
\end{document}